\documentclass[11pt]{article}
\usepackage{amsmath,amssymb,amsthm,amscd}

\newtheorem{theorem}{Theorem}[section]
\newtheorem{lemma}[theorem]{Lemma}
\newtheorem{proposition}[theorem]{Proposition}

\theoremstyle{plain}

\theoremstyle{definition}
\newtheorem{definition}[theorem]{Definition}

\numberwithin{equation}{section}

\renewcommand{\labelenumi}{\textup{(\theenumi)}}

\newcommand{\Homeo}{\operatorname{Homeo}}

\newcommand{\id}{\operatorname{id}}
\newcommand{\Ker}{\operatorname{Ker}}

\newcommand{\dom}{\operatorname{dom}}
\newcommand{\ran}{\operatorname{ran}}
\newcommand{\Ad}{\operatorname{Ad}}

\newcommand{\K}{\mathcal{K}}
\newcommand{\C}{\mathcal{C}}

\newcommand{\N}{\mathbb{N}}

\newcommand{\T}{\mathbb{T}}
\newcommand{\Z}{\mathbb{Z}}
\newcommand{\Zp}{{\mathbb{Z}}_+}

\def\OFL{{\mathcal{O}}_{\frak L}}

\def\DFL{{\mathcal{D}}_{\frak L}}

\def\GNOD{{ G_{N(\OFL,\DFL)} }}

\title{A groupoid approach to $C^*$-algebras associated with $\lambda$-graph systems
and continuous orbit equivalence of subshifts
}
\author{Kengo Matsumoto \\
Department of Mathematics \\
Joetsu University of Education \\
Joetsu, 943-8512, Japan
}
\begin{document}
\maketitle

\date{}

\def\det{{{\operatorname{det}}}}

\begin{abstract} 
A $\lambda$-graph system $\frak L$ 
is a labeled Bratteli diagram with shift operation.
It is a generalized notion of finite labeled graph and presents a subshift. 
We will study continuous orbit equivalence of one-sided subshifts and 
topological conjugacy of two-sided subshifts from the view points of groupoids
and $C^*$-algebras constructed from $\lambda$-graph systems.
\end{abstract}



\def\V{{\mathcal{V}}}
\def\E{{\mathcal{E}}}

\def\OA{{{\mathcal{O}}_A}}
\def\OB{{{\mathcal{O}}_B}}
\def\OZ{{{\mathcal{O}}_Z}}
\def\OTA{{{\mathcal{O}}_{\tilde{A}}}}
\def\SOA{{{\mathcal{O}}_A}\otimes{\mathcal{K}}}
\def\SOB{{{\mathcal{O}}_B}\otimes{\mathcal{K}}}
\def\SOZ{{{\mathcal{O}}_Z}\otimes{\mathcal{K}}}
\def\SOTA{{{\mathcal{O}}_{\tilde{A}}\otimes{\mathcal{K}}}}
\def\DA{{{\mathcal{D}}_A}}
\def\DB{{{\mathcal{D}}_B}}
\def\DZ{{{\mathcal{D}}_Z}}
\def\DTA{{{\mathcal{D}}_{\tilde{A}}}}
\def\SDA{{{\mathcal{D}}_A}\otimes{\mathcal{C}}}
\def\SDB{{{\mathcal{D}}_B}\otimes{\mathcal{C}}}
\def\SDZ{{{\mathcal{D}}_Z}\otimes{\mathcal{C}}}
\def\SDTA{{{\mathcal{D}}_{\tilde{A}}\otimes{\mathcal{C}}}}
\def\BC{{{\mathcal{B}}_C}}
\def\BD{{{\mathcal{B}}_D}}
\def\OAG{{\mathcal{O}}_{A^G}}
\def\OBG{{\mathcal{O}}_{B^G}}
\def\Max{{{\operatorname{Max}}}}
\def\Per{{{\operatorname{Per}}}}
\def\PerB{{{\operatorname{PerB}}}}
\def\Homeo{{{\operatorname{Homeo}}}}
\def\HA{{{\frak H}_A}}
\def\HB{{{\frak H}_B}}
\def\HSA{{H_{\sigma_A}(X_A)}}
\def\Out{{{\operatorname{Out}}}}
\def\Aut{{{\operatorname{Aut}}}}
\def\Ad{{{\operatorname{Ad}}}}
\def\Inn{{{\operatorname{Inn}}}}
\def\det{{{\operatorname{det}}}}
\def\exp{{{\operatorname{exp}}}}
\def\cobdy{{{\operatorname{cobdy}}}}
\def\Ker{{{\operatorname{Ker}}}}
\def\ind{{{\operatorname{ind}}}}
\def\id{{{\operatorname{id}}}}
\def\supp{{{\operatorname{supp}}}}
\def\ActA{{{\operatorname{Act}_{\DA}(\mathbb{T},\OA)}}}
\def\ActB{{{\operatorname{Act}_{\DB}(\mathbb{T},\OB)}}}
\def\RepOA{{{\operatorname{Rep}(\mathbb{T},\OA)}}}
\def\RepDA{{{\operatorname{Rep}(\mathbb{T},\DA)}}}
\def\RepDB{{{\operatorname{Rep}(\mathbb{T},\DB)}}}
\def\U{{{\mathcal{U}}}}

\def\lgs{$\lambda$-graph system}
\def\lgss{$\lambda$-graph systems}
\def\lgss{$\lambda$-graph subsystem}
\def\sms{symbolic matrix system}
\def\smss{symbolic matrix systems}

\section{Introduction}
Interplay between symbolic dynamics  and $C^*$-algebras has been initiated by 
J. Cuntz in \cite{Cuntz} and Cuntz--Krieger in \cite{CK} (see also \cite{Cu3}).
In the former paper, 
Cuntz introduced  a family  $\mathcal{O}_N, N=2,3,\dots$ of $C^*$-algebras
from full $N$-shifts. 
In the latter paper, 
Cuntz--Krieger introduced a family $\OA$ of $C^*$-algebras from general
topological Markov shifts $(\Lambda_A, {\sigma}_A)$ 
defined by $N\times N$ square matrices $A$
with entries in $\{0,1\}$,
as a generalization of Cuntz algebras 
$\mathcal{O}_N.$
The Cuntz algebras $\mathcal{O}_N$ and Cuntz--Krieger algebras $\OA$
have played extremely important r\^{o}le of classification and structure theory of $C^*$-algebras.
Cuntz--Krieger in \cite{CK} 
among other things showed close relationship
between topological Markov shifts and $C^*$-algebras 
by proving several fundamental results below.
Let us denote by $\rho^A$ the gauge action on $\OA$ and
$\DA$ the canonical diagonal $C^*$-subalgebra of the AF-subalgebra
$(\OA)^{\rho}$ 
of $\OA$ consisting of elements fixed by  the action $\rho^A$. 
We henceforth denote by $\K$ the $C^*$-algebra 
of compact operators on  separable infinite dimensional Hilbert space $\ell^2(\N)$ and
by $\C$ the maximal abelian $C^*$-subalgebra of $\K$ 
consisting of diagonal operators on $\ell^2(\N)$.
An irreducible square matrix with entries in $\{0,1\}$
is said to satisfy condition (I) if it is not a permutation matrix (\cite{CK}). 
\begin{theorem}[Cuntz--Krieger {\cite{CK}}]
Let $A, B$ be irreducible square matrices 
with entries  in $\{0,1\}$ satisfying condition (I).
\begin{enumerate}
\renewcommand{\theenumi}{\roman{enumi}}
\renewcommand{\labelenumi}{\textup{(\theenumi)}}
\item
If the one-sided topological Markov shifts $(X_A,\sigma_A)$ and 
$(X_B,\sigma_B)$ are topologically conjugate,
then there exists an isomorphism
$\Phi:\OA\longrightarrow \OB$ of the Cuntz--Krieger algebras such that 
$\Phi(\DA) = \DB$ and 
$\Phi\circ\rho^A_t = \rho^B_t\circ\Phi, t \in \T$,
\item   
If the two-sided topological Markov shifts $(\Lambda_A,{\sigma}_A)$ and 
$(\Lambda_B, {\sigma}_B)$ are topologically conjugate,
then there exists an isomorphism
$\bar{\Phi}:\SOA\longrightarrow \SOB$ of the stabilized Cuntz--Krieger algebras such that 
$\bar{\Phi}(\SDA) = \SDB$ and 
$\bar{\Phi}\circ(\rho^A_t\otimes\id) = (\rho^B_t\circ\otimes\id)\circ\bar{\Phi}, t \in \T$,
\item   
If the two-sided topological Markov shifts $(\Lambda_A,{\sigma}_A)$ and 
$(\Lambda_B,{\sigma}_B)$ are flow equivalent, 
then there exists an isomorphism
$\Psi:\SOA\longrightarrow \SOB$ of the stabilized Cuntz--Krieger algebras such that 
$\Psi(\SDA) = \SDB$.
\end{enumerate}
\end{theorem}
The author in \cite{MaPacific} introduced 
a notion of continuous orbit equivalence of one-sided topological Markov shifts.
By using the notion, 
the author proved that   
the one-sided topological Markov shifts $(X_A,\sigma_A)$ and 
$(X_B,\sigma_B)$ are eventually conjugate if and only if there exists an isomorphism
$\Phi:\OA\longrightarrow \OB$ of the Cuntz-Krieger algebras such that 
$\Phi(\DA) = \DB$ and 
$\Phi\circ\rho^A_t = \rho^B_t\circ\Phi, t \in \T$.
Recently, 
K. A. Brix and T.  M. Carlsen in \cite{BC}
 characterized one-sided topological conjugate Markov shifts in terms of its groupoids 
and the Cuntz-Krieger algebras with gauge actions.
T. M. Carlsen  and J. Rout in \cite{CR} proved the converse of the above (ii)
for more general graph algebras by using groupoid technique
studied in earlier works (\cite{AER}, \cite{CRS}, etc.).
Concerning flow equivalence, 
H. Matui and the author proved the converse of (iii) 
(\cite{MMKyoto}, see \cite{CEOR} for more general matrices).

In this paper, 
we will study generalization of the above Cuntz--Krieger's Theorem 
to general subshifts.
To construct a $C^*$-algebra from a general subshift, 
we have used in \cite{MaDocMath2002} 
a graphical object called $\lambda$-graph system.
A $\lambda$-graph system 
${\frak L}=(V, E,\lambda,\iota)$ over a finite  alphabet $\Sigma$
consists of a vertex set $V =\bigcup_{l=0}^\infty V_l$, 
an edge set 
$E =\bigcup_{l=0}^\infty E_{l,l+1}$, 
a labeling map
$\lambda:E\longrightarrow\Sigma$ and 
a surjection 
$\iota(=\iota_{l,l+1}):V_{l+1}\longrightarrow V_l, l \in \Zp$ 
such that 
$(V,E,\lambda)$ is a labeled Bratteli diagram and $\iota$ 
plays a r{\^o}le of  shift on 
the Bratteli diagram.
A $\lambda$-graph system is said to be left-resolving
if two edges $e,f \in E_{l,l+1}$ have the same terminal vertex $t(e) = t(f)$
and the same label $\lambda(e) = \lambda(f),$
then $e=f.$
A finite directed  labeled graph naturally yields a $\lambda$-graph system,
and any $\lambda$-graph system presents a subshift.
Conversely any subshift may be presented by a left-resolving $\lambda$-graph system.
Let $(\Lambda,\sigma_\Lambda)$ 
denote the presented two-sided subshift
 and $(X_\Lambda,\sigma_\Lambda)$  the presented one-sided subshift
by a $\lambda$-graph system ${\frak L}$.
A $\lambda$-graph system 
${\frak L}=(V, E,\lambda,\iota)$ over an alphabet $\Sigma$
gives rise to a topological dynamical system
$(X_{\frak L},\sigma_{\frak L})$ 
and a continuous surjection
$\pi_{\frak L}:X_{\frak L} \longrightarrow X_\Lambda$
such that 
$\sigma_\Lambda\circ\pi_{\frak L} = \pi_{\frak L}\circ \sigma_{\frak L}.$
If ${\frak L}$ is defined by a finite labeled graph,
the presented subshift becomes a sofic shift and the surjection $\pi_{\frak L}$ exactly corresponds to 
the Markov cover for the sofic shift.
The topological dynamical system
$(X_{\frak L},\sigma_{\frak L})$ 
naturally yields an \'etale groupoid 
$G_{\frak L}$  and a continuous graph $E_{\frak L}$
in the sense of V. Deaconu (\cite{MaDocMath2002}, cf. \cite{De}, \cite{De2}).

In \cite{MaDocMath2002}, the author introduced a $C^*$-algebra $\OFL$
from $\lambda$-graph system ${\frak L}$
as a generalization of Cuntz--Krieger algebras.
The $C^*$-algebra was defined to be the  $C^*$-algebra $C^*(G_{\frak L})$ 
of the groupoid $G_{\frak L}$. 
It has a gauge action $\rho^{\frak L}$ of $\T$ and more generally
a generalized gauge action $\rho^{{\frak L},f}$ on it for a continuous homomorphism
$f:G_{\frak L}\longrightarrow \Z$ of groupoids.
As in \cite{MaDocMath2002}, 
the fixed point algebra $(\OFL)^{\rho^{\frak L}}$
of
$\OFL$ under the gauge action
$\rho^{\frak L}$ is an AF-algebra 
whose diagonal subalgebra denoted by $\DFL$ is canonically isomorphic to the commutative 
$C^*$-algebra $C(X_{\frak L})$ 
of continuous functions on $X_{\frak L}$.
A $\lambda$-graph system ${\frak L}$ is said to be essentially free 
if the topological dynamical system
$(X_{\frak L},\sigma_{\frak L})$ 
is essentially free.
When this is the case, 
the subalgebra $\DFL$  is maximal abelian in $\OFL.$ 
If a $\lambda$-graph system satisfies condition (I) in the sense of 
\cite{MaDocMath2002}, it is essentially free.
There is a natural surjection 
$\pi_{\frak L}: X_{\frak L} \longrightarrow X_{\Lambda}$ 
which induces an embedding 
$C(X_{\Lambda})\hookrightarrow C(X_{\frak L})$.
The algebra $C(X_\Lambda)$ denoted by
$\mathcal{D}_\Lambda$ 
 is regarded as a subalgebra of $\DFL$.

Continuous orbit equivalence between general one-sided subshifts has been first studied in \cite{MaYMJ2010}(cf. \cite{MaCM2009})
in terms of factor maps $\pi_{\frak L}: X_{\frak L} \longrightarrow X_\Lambda.$
The continuous orbit equivalence between two factor maps 
$\pi_{{\frak L}_1}: X_{{\frak L}_1} \longrightarrow X_{\Lambda_1}$ 
and 
$\pi_{{\frak L}_2}: X_{{\frak L}_2} \longrightarrow X_{\Lambda_2}$ 
will be called  $({\frak L}_1,{\frak L}_2)$-continuously orbit equivalent in Section 4. 
In this paper, we will study continuous orbit equivalence between subshifts, 
conjugacy of the one-sided and the  two-sided topological dynamical systems 
$({X}_{\frak L},{\sigma}_{\frak L})$ and
$(\bar{X}_{\frak L},\bar{\sigma}_{\frak L})$
and the associated $C^*$-algebras $\OFL$ from the view points of groupoids. 
The discussions in this paper are inspired by Arklint--Eilers--Ruiz's paper
\cite{AER} and Carlsen--Rout's paper \cite{CR}.
We will first show the following theorem. 
\begin{theorem}\label{thm:onesidedcoe}
Let ${\frak L}_1$ and 
 ${\frak L}_2$ be left-resolving $\lambda$-graph systems satisfying condition (I).
Let $(X_{\Lambda_1}, \sigma_{\Lambda_1})$ 
and 
$(X_{\Lambda_2}, \sigma_{\Lambda_2})$
be the one-sided subshifts presented by  ${\frak L}_1$ and 
 ${\frak L}_2$, respectively.
Then the  following three assertions are equivalent:
\begin{enumerate}
\renewcommand{\theenumi}{\roman{enumi}}
\renewcommand{\labelenumi}{\textup{(\theenumi)}}
\item One-sided subshifts 
$(X_{\Lambda_1},\sigma_{\Lambda_1})$
and
$(X_{\Lambda_2},\sigma_{\Lambda_2})$
are 
$({\frak L}_1,{\frak L}_2)$-continuously orbit equivalent.
\item
There exist an isomorphism
$\varphi:G_{{\frak L}_1}\longrightarrow G_{{\frak L}_2}$
of \'etale groupoids and a homeomorphism
$h:X_{\Lambda_1}\longrightarrow X_{\Lambda_2}$
of the shift spaces such that
$\pi_{{\frak L}_2} \circ\varphi|_{X_{{\frak L}_1}} = h\circ\pi_{{\frak L}_1}$,
where $X_{{\frak L}_i}$ is identified with
the unit space $G_{{\frak L}_i}^0$ of the \'etale groupoid $G_{{\frak L}_i}, i=1,2.$
\item
There exists an isomorphism
$\Phi:{\mathcal{O}}_{{\frak L}_1}\longrightarrow {\mathcal{O}}_{{\frak L}_2}
$ of $C^*$-algebras
such that  
$\Phi({\mathcal{D}}_{\Lambda_1})={\mathcal{D}}_{\Lambda_2}.
$ 
\end{enumerate}
\end{theorem}
There exists a natural cocycle function
$c_{\frak L}:G_{\frak L}\longrightarrow \Z$
of the \'etale groupoid $G_{\frak L}$.
Eventual conjugacy of one-sided topological Markov shifts 
will be generalized to one-sided subshifts with $\lambda$-graph systems as 
$({\frak L}_1,{\frak L}_2)$-eventual conjugacy, 
so that we will prove the following theorem. 
\begin{theorem}\label{thm:eventconj}
Let ${\frak L}_1$ and 
 ${\frak L}_2$ be left-resolving $\lambda$-graph systems satisfying condition (I).
Let 
$(X_{\Lambda_1},\sigma_{\Lambda_1})$ and 
$(X_{\Lambda_2},\sigma_{\Lambda_2})$
be the one-sided subshifts presented by  ${\frak L}_1$ and 
 ${\frak L}_2$, respectively.
Then the following  three assertions are equivalent:
\begin{enumerate}
\renewcommand{\theenumi}{\roman{enumi}}
\renewcommand{\labelenumi}{\textup{(\theenumi)}}
\item One-sided subshifts 
$(X_{\Lambda_1},\sigma_{\Lambda_1})$
and
$(X_{\Lambda_2},\sigma_{\Lambda_2})$
are 
$({\frak L}_1,{\frak L}_2)$-eventually conjugate.
\item
There exist an isomorphism
$\varphi:G_{{\frak L}_1}\longrightarrow G_{{\frak L}_2}$
of \'etale groupoids and a homeomorphism
$h:X_{\Lambda_1}\longrightarrow X_{\Lambda_2}$
of the shift spaces such that
$\pi_{{\frak L}_2} \circ\varphi|_{X_{{\frak L}_1}} = h\circ\pi_{{\frak L}_1}$
and
$c_{{\frak L}_2} \circ \varphi = c_{{\frak L}_1}$. 
\item
There exists an isomorphism
$\Phi:{\mathcal{O}}_{{\frak L}_1}\longrightarrow {\mathcal{O}}_{{\frak L}_2}
$ 
of $C^*$-algebras such that 
\begin{equation*}
\Phi({\mathcal{D}}_{{\Lambda}_1})={\mathcal{D}}_{{\Lambda}_2}
\quad
\text{ and }
\quad
\Phi \circ \rho_t^{{\frak L}_1} =\rho_t^{{\frak L}_2}\circ \Phi,
\quad
t \in \T.
\end{equation*} 
\end{enumerate}
\end{theorem}
For a left-resolving $\lambda$-graph system
${\frak L} =(V,E,\lambda,\iota)$ over $\Sigma$,
we will construct
its stabilization
$\widetilde{\frak L} =(\widetilde{V},\widetilde{E}, \tilde{\lambda},\tilde{\iota})$
and its groupoid $G_{\widetilde{\frak L}}$
such that 
its groupoid $C^*$-algebra
$C^*(G_{\widetilde{\frak L}})$ which is written
${\mathcal{O}}_{\widetilde{\frak L}}$, 
its canonical maximal abelian $C^*$-subalgebra
${\mathcal{D}}_{\widetilde{\frak L}}$,
and its gauge action
$\rho^{\widetilde{\frak L}}$
are isomorphic to
$\OFL\otimes\K$,
$\DFL\otimes\C$,
and
$\rho^{\frak L}\otimes\id$,
respectively.
The idea of the construction of 
$\widetilde{\frak L}$ is essentially due to the Tomforde's construction 
for directed graphs in 
\cite{Tomforde2}.
We will introduce a notion of  $({\frak L}_1,{\frak L}_2)$-conjugacy 
between two-sided subshifts $\Lambda_1$ and $\Lambda_2$,
that is regarded as a topological conjugacy 
between $\Lambda_1$ and $\Lambda_2$
relative to the factor maps 
$\bar{\pi}_{{\frak L}_1}: \bar{X}_{{\frak L}_1} \longrightarrow \Lambda_1$ and
$\bar{\pi}_{{\frak L}_2}: \bar{X}_{{\frak L}_2} \longrightarrow \Lambda_2$,
where  $\bar{\pi}_{{\frak L}_i}, i=1,2,$ is the two-sided extensions of
${\pi}_{{\frak L}_i}, i=1,2$.
We will show the following theorem  concerning two-sided subshifts.
\begin{theorem}[Proposition \ref{prop:6.2} and 
Theorem \ref{thm:twosidedconjugacy1}] \label{thm:1.4}
Let ${\frak L}_1$ and  ${\frak L}_2$ be left-resolving $\lambda$-graph systems satisfying condition (I).
Let
$(\Lambda_1,\sigma_1)$ and $(\Lambda_2,\sigma_2)$
be their two-sided subshifts presented by 
${\frak L}_1$ and ${\frak L}_2$, respectively.
Then the following three assertions are equivalent:
 \begin{enumerate}
\renewcommand{\theenumi}{\roman{enumi}}
\renewcommand{\labelenumi}{\textup{(\theenumi)}}
\item
The two-sided subshifts 
$(\Lambda_1, \sigma_1)$ and 
$(\Lambda_2, \sigma_2)$
are $({\frak L}_1,{\frak L}_2)$-conjugate.
\item
There exist an isomorphism
$\tilde{\varphi}: G_{\widetilde{\frak L}_1}\longrightarrow 
G_{\widetilde{\frak L}_2}
$ of \'etale groupoids 
and a homeomorphism
$\tilde{h} :\widetilde{X}_{\Lambda_1} \longrightarrow \widetilde{X}_{\Lambda_2}
$
such that 
$\tilde{\pi}_{{\frak L}_2} 
\circ \tilde{\varphi}|_{G_{\widetilde{\frak L}_1}^0} = 
\tilde{h}\circ \tilde{\pi}_{{\frak L}_1}
$ 
and 
$c_{\widetilde{\frak L}_2}\circ\tilde{\varphi} =
c_{\widetilde{\frak L}_1},$
where $\widetilde{X}_{\Lambda_i}$ is the stabilization of 
$X_{\Lambda_i}, i=1,2$.  
\item
There exists an isomorphism
$\widetilde{\Phi}:{\mathcal{O}}_{{\frak L}_1}\otimes\K
\longrightarrow 
{\mathcal{O}}_{{\frak L}_2}\otimes\K
$ of $C^*$-algebras
such that 
\begin{equation*}
\widetilde{\Phi}({\mathcal{D}}_{{\Lambda}_1}\otimes\C)
={\mathcal{D}}_{{\Lambda}_2}\otimes\C, \qquad
\widetilde{\Phi}\circ (\rho^{{\frak L}_1}_t\otimes\id)
 =(\rho^{{\frak L}_2}_t\otimes\id)\circ \widetilde{\Phi},\quad t \in \T.
\end{equation*} 
\end{enumerate}
\end{theorem}

\medskip

Throughout the paper, 
we denote by
$\Zp$ the set of nonnegative integers
and by 
$\N$ the set of positive integers,
respectively.

\section{Preliminaries}
Let $\Sigma$ be a finite set, which we call an alphabet. 
Let ${\frak L} =(V,E,\lambda,\iota)$ 
be 
a $\lambda$-graph system over $\Sigma$ with vertex set
$
V = \bigcup_{l=0}^\infty V_{l}
$
and  edge set
$
E = \bigcup_{l=0}^\infty E_{l,l+1}
$
that is labeled with symbols in $\Sigma$ by $\lambda: E \rightarrow \Sigma$, 
and that is supplied with  surjective maps
$
\iota( = \iota_{l,l+1}):V_{l+1} \rightarrow V_l
$
for
$
l \in  \Zp.
$
Here the vertex sets $V_{l},l \in \Zp$
are finite disjoint sets,
as well as   
$E_{l,l+1},l \in \Zp$
are finite disjoint sets.
An edge $e$ in $E_{l,l+1}$ has its source vertex $s(e)$ in $V_{l}$ 
and its terminal  vertex $t(e)$ 
in
$V_{l+1}$
respectively.
Every vertex in $V$ has a successor and  every 
vertex in $V_l$ for $l\in \N$ has a predecessor. 
It is then required for definition of $\lambda$-graph system that there exists an edge in $E_{l,l+1}$
with label $\alpha$ and its terminal is  $v \in V_{l+1}$
 if and only if 
 there exists an edge in $E_{l-1,l}$
with label $\alpha$ and its terminal is $\iota(v) \in V_{l}.$
For 
$u \in V_{l-1}$ and
$v \in V_{l+1},$
we put
\begin{align*}
E^{\iota}(u, v)
& = \{e \in E_{l,l+1} \ | \ \iota(s(e)) = u, \, t(e) = v \},\\
E_{\iota}(u, v)
& = \{e \in E_{l-1,l} \ | \ s(e) = u, \, t(e) = \iota(v) \}.
\end{align*}
As a key hypothesis for ${\frak L}$ to be a $\lambda$-graph system, 
we require the condition 
that there exists a bijective correspondence between 
$
E^{\iota}(u, v)
$
and
$
E_{\iota}(u, v)
$
that preserves labels
for each pair $(u,v) \in V_{l-1}\times V_{l+1}$ 
of vertices.
We call this property  the local property of $\lambda$-graph system. 
A $\lambda$-graph system ${\frak L}$ 
is said to be left-resolving if $e,f \in E$ with
$t(e) = t(f)$ and $\lambda(e) = \lambda(f)$ implies $e=f$.
In what follows all $\lambda$-graph systems are assumed to be left-resolving.

Let us recall the construction of the topological dynamical system
$(X_{\frak L},\sigma_{\frak L})$,
the \'etale groupoid $G_{\frak L}$
and the $C^*$-algebra $\OFL$
by following \cite{MaDocMath2002}.
Let $\Omega_{\frak L}$
be  the projective limit of the system
$\iota_{l,l+1}: V_{l+1}\rightarrow V_l, l \in \Zp,$ that is defined by
\begin{equation*}
\Omega_{\frak L} = \{ (u^l)_{l \in \Zp}\in 
\prod_{l \in \Zp} V_l
 \ | \ \iota_{l,l+1}(u^{l+1}) = u^l, l\in \Zp \}.
\end{equation*}
We endow $\Omega_{\frak L}$  with the projective limit topology 
so that it is a compact Hausdorff space.
An element $v$ in $\Omega_{\frak L}$
is called  an $\iota$-orbit
or also a vertex.
Let $E_{\frak L}$ 
be the set of all triplets 
$(u, \alpha,w) \in \Omega_{\frak L} \times \Sigma \times \Omega_{\frak L}$
such that
for each $l \in \Zp$, 
there exists $e_{l,l+1} \in E_{l,l+1}$ satisfying
\begin{equation*}
u^l = s(e_{l,l+1}),\quad
  w^{l+1} = t(e_{l,l+1}) \quad \text{ and } \quad
  \alpha = \lambda(e_{l,l+1})
\end{equation*}
where $u = (u^l)_{l \in \Zp}, w = (w^l)_{l \in \Zp}\in \Omega_{\frak L}.$
The set
$
E_{\frak L}\subset \Omega_{\frak L} \times \Sigma \times \Omega_{\frak L}
$
becomes a  zero-dimensional continuous graph
in the sense of V. Deaconu
(\cite[Proposition 2.1]{MaDocMath2002}, cf. \cite{De}, \cite{De2}).

Let us denote by
$\{{v}_1^l,\dots,{v}_{m(l)}^l \}$
the vertex set
$V_l.$ 
Define a clopen set $U_i^1(\alpha)$ in $E_{\frak L}$ 
for 
$\alpha \in \Sigma,i=1,\dots,m(1)$
by 
\begin{equation*}
U_i^1(\alpha) = 
\{(u,\alpha,w) \in E_{\frak L} \mid 
 w^1 = {v}_i^1
 \text{ where }
 w = (w^l)_{l\in \Zp} \in \Omega_{\frak L} \}
\end{equation*}
so that
\begin{equation*}
\bigcup_{\alpha \in \Sigma}\bigcup_{i=1}^{m(1)}U_i^1(\alpha) = E_{\frak L},
\qquad
U_i^1(\alpha) \bigcap U_j^1(\beta) = \emptyset
\quad
\text{ if } \
(\alpha,i) \ne (\beta,j).
\end{equation*}
Put
$t_{\frak L}(u,\alpha,w) = w$ for 
$(u,\alpha,w) \in E_{\frak L}.$
Since
${\frak L}$ is left-resolving,
the restriction of $t_{\frak L}$ to $U_i^1(\alpha)$
is a homeomorphism onto 
$
U_{ {v}_i^1} 
= \{ (w^l)_{l \in \Zp} \in  \Omega_{\frak L} \mid w^1 = {v}_i^1 \}
$
if $U_i^1(\alpha) \ne \emptyset$.
Hence
$t_{\frak L}:E_{\frak L} \rightarrow \Omega_{\frak L}$
is a local homeomorphism.
Let $X_{\frak L}$ 
 be the set of all one-sided paths of $E_{\frak L}$:
\begin{align*}
X_{\frak L}= \{ (\alpha_i,u_i)_{i\in\N} \in 
\prod_{i\in\N}
(\Sigma \times \Omega_{\frak L}) \ | \
& (u_{0},\alpha_1,u_{1}) \in E_{\frak L}
 \text{ for some } u_0 \in \Omega_{\frak L}
 \\ 
 \text{ and } 
& (u_{i},\alpha_{i+1},u_{i+1}) \in E_{\frak L} 
 \text{ for all } i\in \N \}.
\end{align*}
The set
$X_{\frak L}$ has the relative topology from the infinite product topology
of $\Pi_{i\in \N}(\Sigma \times \Omega_{\frak L})$. 
It is a zero-dimensional compact Hausdorff space.   
For $\mu=(\mu_1,\dots,\mu_k) \in \Sigma^k, v_i^l \in V_l$
with $k\le l$,
define a subset $U(\mu,v_i^l)$ of $X_{\frak L}$ by 
\begin{equation}
U(\mu,v_i^l) =\{
(\alpha_i,u_i)_{i\in\N} \in X_{\frak L}\mid
\alpha_1=\mu_1,\dots,\alpha_k=\mu_k, \, u_k^l = v_i^l\} \label{eq:muv}
\end{equation}
where 
$u_i =(u_i^l)_{l\in \N} \in \Omega_{\frak L}$.
The set
$U(\mu,v_i^l)$ is clopen and such family generate the topology of $X_{\frak L}$.

The shift map 
$
\sigma_{\frak L} :(\alpha_i,u_i)_{i\in\N}\in X_{\frak L} 
\rightarrow 
(\alpha_{i+1},u_{i+1})_{i\in\N}\in
X_{\frak L}
$
is continuous.
For $v = (v^l)_{l \in \Zp} \in \Omega_{\frak L}$
and $ \alpha \in \Sigma $,
the local  property of $\lambda$-graph system $\frak L$ ensures  that if 
there exists $e_{0,1} \in E_{0,1}$ satisfying
$
v^1 = t(e_{0,1}), \alpha = \lambda(e_{0,1}),
$
there exist $e_{l,l+1} \in E_{l,l+1}$ and 
$u =(u^l)_{l\in \Zp} \in \Omega_{\frak L}$
 satisfying
$u^l = s(e_{l,l+1}),
  v^{l+1} = t(e_{l,l+1}),
    \alpha = \lambda(e_{l,l+1})
$
for each $l \in \Zp.$
Hence 
for any $x = (\alpha_i,v_i)_{i\in\N} \in X_{\frak L}$,
there uniquely exists $v_0 \in \Omega_{\frak L}$ such that
$(v_0,\alpha_1,v_1) \in E_{\frak L}$.
Denote by 
$v(x)_0$ 
the  unique vertex $v_0$ for $x \in X_{\frak L}.$  
We are always assuming that 
the $\lambda$-graph system ${\frak L}$ is left-resolving,
so that 
 $\sigma_{\frak L}$ is a local homeomorphism on $X_{\frak L}$
(\cite[Lemma 2.2]{MaDocMath2002}).
 Define 
$\pi_{\frak L}:(\alpha_i,u_i)_{i\in\N}\in X_{\frak L}
\longrightarrow 
(\alpha_i)_{i\in \N} \in \Sigma^\N.
$
The image
$\pi_{\frak L}(X_{\frak L})$ in $\Sigma^\N$
is denoted by 
$X_\Lambda$,
which is the shift space of the one-sided subshift
denoted by $(X_\Lambda,\sigma_\Lambda)$
with shift transformation
$\sigma_\Lambda((\alpha_i)_{i\in \N}) =(\alpha_{i+1})_{i\in \N}.
$ 
Hence we have 
$\pi_{\frak L}\circ\sigma_{\frak L}
= \sigma_\Lambda\circ\pi_{\frak L}.$

Let us construct the $C^*$-algebra
 $\OFL$ following \cite{MaDocMath2002}.
Following V. Deaconu \cite{De}, \cite{De2}
and J. Renault \cite{Renault}, \cite{Renault2}, \cite{Renault3}, 
one may construct a locally compact \'etale groupoid 
$G_{\frak L}$, called a Renault--Deaconu groupoid,  from a local  homeomorphism 
$\sigma_{\frak L}$ on $X_{\frak L}$
 as in the following way.
We put
\begin{equation*}
G_{\frak L} = 
\{ (x,n,z) \in X_{\frak L} \times {\Bbb Z} \times X_{\frak L} \ | \
\text{ there exist } k,l\in \Zp ; \ \sigma^k_{\frak L}(x) = \sigma^l_{\frak L}(z), 
n=k-l \}.
\end{equation*}
The range map and the domain  map are defined by
$$
r(x,n,z) = x,\qquad d(x,n,z) =z
\quad
\text{ for }
\quad
(x,n,z) \in G_{\frak L}.
$$
The multiplication and the inverse operation are defined by
$$
(x,n,z)(z,m,w) = (x,n+m,w),\qquad
(x,n,z)^{-1} = (z,-n,x).
$$
The unit space
$
G_{\frak L}^0 
= \{ (x,0,x) \in G_{\frak L} \ | \ x \in X_{\frak L} \}
$
is identified with  the space
$
 X_{\frak L}
$
through the map
$x \in X_{\frak L}\longrightarrow 
(x,0,x) \in G_{\frak L}^0.
$
A basis of open sets of $G_{\frak L}$ is given by
$$
Z(U,k,l,V) 
= \{ (x,k-l,z) \in G_{\frak L} \ | \  x \in U, z \in V, 
\sigma^k_{\frak L}(x) = \sigma^l_{\frak L}(z) \}
$$
where
$U,V$ are open sets of $X_{\frak L}$,
 and $k,l$ 
are nonnegative integers such that 
$\sigma^k|_{U}$ and $\sigma^l|_{V}$ 
are homeomorphisms with the same open range.
The groupoid $C^*$-algebra $C^*(G_{\frak L})$ 
for the groupoid $G_{\frak L}$ is defined as in the following way 
(\cite{Renault}, \cite{Renault2}, \cite{Renault3}, cf. \cite{De}, \cite{De2}).
Let $C_c(G_{\frak L}) $ be the set of all continuous functions on
$G_{\frak L}$ with compact support that has a natural product structure and $*$-involution of 
$*$-algebra given by
\begin{align*}
(f*g)(s) &  = 
 \sum_{t \in G_{\frak L}, \ r(t) = r(s)} f(t) g(t^{-1}s) 
           = 
 \sum_{t_1,t_2  \in G_{\frak L},\ s = t_1 t_2} f(t_1) g(t_2), \\
  f^*(s) & = \overline{f(s^{-1})}, 
  \qquad f,g \in C_c(G_{\frak L}), \quad s \in G_{\frak L}.     
\end{align*}
Let $C_0(G_{\frak L}^0) $ be the $C^*$-algebra  of all continuous functions on
$G_{\frak L}^0$ that vanish at infinity. 
The algebra
$C_c(G_{\frak L}) $
is a 
$C_0(G_{\frak L}^0) $-right module, endowed with 
a $C_0(G_{\frak L}^0) $-valued inner product defined by
\begin{align}
 (\xi f )(x,n,z) 
 = & \xi(x,n,z)f(z), 
 \qquad
   \xi \in C_c(G_{\frak L}), \quad f \in C_0(G_{\frak L}^0),
   \quad (x,n,z) \in G_{\frak L},
 \\
  < \xi, \eta >(z) 
 = & \sum_{ 
 { (x,n,z) \in G_{\frak L} }}
   \overline{ \xi (x,n,z)} \eta (x,n,z),
   \qquad
   \xi,\eta \in C_c(G_{\frak L}), \quad z \in X_{\frak L}.
 \end{align}
 Let us denote by $l^2(G_{\frak L})$ the completion of the inner product 
$C_0(G_{\frak L}^0) $-right module
$C_c(G_{\frak L})$.
It is a Hilbert $C^*$-right module over the commutative $C^*$-algebra 
$C_0(G_{\frak L}^0) $.
We denote by
$B(l^2(G_{\frak L}))$
the $C^*$-algebra of all bounded adjointable
$C_0(G_{\frak L}^0) $-module maps on
$l^2(G_{\frak L}).$
Let $\pi $ be the $*$-homomorphism of 
$C_c(G_{\frak L})$ into
$B(l^2(G_{\frak L}))$
defined by
$\pi (f)\xi = f * \xi$
for
$f, \xi \in 
C_c(G_{\frak L}).$
Then the closure of $\pi (C_c(G_{\frak L}))$ in
$B(l^2(G_{\frak L}))$
is called the (reduced) $C^*$-algebra of the groupoid $G_{\frak L}$,
that we denote by
$C^*_r(G_{\frak L}).$
General theory of $C^*$-algebras of groupoids says that 
for a Renault--Deaconu groupoid $G$, 
the reduced $C^*$-algebra $C^*_r(G)$ and
the universal $C^*$-algebra $C^*(G)$ 
may be identified, as they are isomorphic,
see for instance (\cite[Proposition 2.4]{Renault2000}).
\begin{definition}[{\cite{MaDocMath2002}}] 
The $C^*$-algebra $\OFL$ associated with  $\lambda$-graph system 
$\frak L$ is defined to be the  $C^*$-algebra 
$C^*_r(G_{\frak L})$
of the groupoid $G_{\frak L}.$
\end{definition}
We will describe the algebraic structure of the $C^*$-algebra $\OFL$. 
Let us denote by $X_\Lambda$ 
the one-sided subshift presented by 
${\frak L}$.
Recall that $B_k(X_\Lambda)$ 
denotes the set of all words of $\Sigma^k$ that appear in $\Lambda$.
For 
$x = (\alpha_n,u_n)_{n\in \N} \in X_{\frak L},$
we put
$\lambda(x)_n = \alpha_n \in \Sigma,$
$v(x)_n = u_n \in \Omega_{\frak L}$
for $n \in \N,$ 
respectively.
The $\iota$-orbit
$v(x)_n$ is written as
$v(x)_n = {(v(x)^l_n)}_{l\in \Zp} 
\in \Omega_{\frak L}.$
Now $\frak L$ is left-resolving
so that there uniquely exists 
$v(x)_0 \in \Omega_{\frak L}$ satisfying
$(v(x)_0,\alpha_1,u_1) \in E_{\frak L}.$
Define 
$U(\mu)\subset G_{\frak L}$ 
for 
$\mu = (\mu_1,\dots,\mu_k) \in B_k(X_\Lambda),$
and
$U(v_i^l)\subset G_{\frak L}$
for $v_i^l \in V_l$ by
\begin{gather*}
 U(\mu) 
 = \{ (x,k,z) \in G_{\frak L} \mid
 \sigma_{\frak L}^k(x) = z,   
\lambda(x)_1 =\mu_1,\dots,\lambda(x)_k =\mu_k \}, \quad \text{ and}\\
 U({v}_i^l) 
 = \{ (x,0,x) \in G_{\frak L} \ | \ v(x)_0^l = {v}_i^l\}
\end{gather*}
where
$v(x)_0 = (v(x)^l_0)_{l\in \Zp}\in \Omega_{\frak L}.$
They are clopen sets of $G_{\frak L}$.
We set
$$
S_{\mu} =\pi( \chi_{U(\mu)} ),\qquad E_i^l = \pi(\chi_{U({v}_i^l) })
\qquad \text{ in }\quad \pi(C_c(G_{\frak L}))
$$
where 
$
\chi_{F}
\in C_c(G_{\frak L})
$ denotes 
the characteristic function of a clopen set
$F$
on the space
$
G_{\frak L}.
$
We in particular write $S_\mu$ as $S_\alpha$
for the symbol $\mu = \alpha \in \Sigma$.
For $\mu =(\mu_1,\dots,\mu_k), \nu=(\nu_1,\dots,\nu_m)\in B_*(X_\Lambda)$
and
$v_i^l\in V_l$
with $k,m\le l$,
put the clopen set of $G_{\frak L}$ by 
\begin{align*}
U(\mu,v_i^l,\nu)
=\{ & (x,n,z) \in G_{\frak L}\mid  
n = k-m, \lambda(x)_{[1,k]} =\mu, \lambda(z)_{[1,m]} = \nu,\, \\
& \sigma_{\frak L}^k(x) = \sigma_{\frak L}^m(z), \, v(x)_k^l = v(z)_m^l =v_i^l\}
\end{align*}
where 
$x=(\lambda(x)_i,v(x)_i)_{i\in \N}, z=(\lambda(z)_i,v(z)_i)_{i\in \N}
\in X_{\frak L}$
with
$
v(x)_i =(v(x)_i^l)_{l\in \Zp},
v(z)_i =(v(z)_i^l)_{l\in \Zp}\in \Omega_{\frak L}
$
and
$\lambda(x)_{[1,k]} =(\lambda(x)_1,\dots,\lambda(x)_k) \in B_k(X_\Lambda).$
Then we have 
\begin{equation*}
S_\mu E_i^l S_\nu^* = \pi(\chi_{U(\mu,v_i^l,\nu)}).
\end{equation*}
In particular, for the clopen set $U(\mu,v_i^l)$ with
$\mu \in (\mu_1,\cdots,\mu_k)\in B_k(X_\Lambda), v_i^l \in V_l$
defined in 
\eqref{eq:muv},
we know that 
\begin{equation}
S_\mu E_i^l S_\mu^* = \pi(\chi_{U(\mu,v_i^l)}). \label{eq:smueilsmu}
\end{equation}
The transition matrices $A_{l,l+1}, I_{l,l+1}$
for  ${\frak{L}}$
are defined by setting
\begin{align*}
A_{l,l+1}(i,\alpha,j)
 & =
{\begin{cases}
1 &  
    \text{ if there exists } e \in E_{l,l+1}; 
 \ s(e) = v_i^l, \lambda(e) = \alpha,
                       t(e) = v_j^{l+1}, \\
0           & \text{ otherwise,}
\end{cases}} \\
I_{l,l+1}(i,j)
 & =
{\begin{cases}
1 &  
    \text{ if } \ \iota_{l,l+1}(v_j^{l+1}) = v_i^l, \\
0           & \text{ otherwise}
\end{cases}} 
\end{align*} 
for
$
i=1,2,\dots,m(l),\ j=1,2,\dots,m(l+1), \ \alpha \in \Sigma.
$ 
We say that ${\frak{L}}$ satisfies condition (I) if for each
$v\in V_l,$
the subset $\Gamma_\infty^{+}(v)$ of  $X_\Lambda$
defined by
\begin{align*}
\Gamma_\infty^{+}(v) 
=\{ (\alpha_1,\alpha_2,\dots, ) \in \Sigma^{\N} &  | 
 \text{ there exists } e_{n,n+1} \in E_{n,n+1} 
 \text{ for } n \ge l ; \\
 v =  s(e_{l,l+1}),\ &
 t(e_{n,n+1})  = s(e_{n+1,n+2}), \ 
 \lambda(e_{n,n+1}) = \alpha_{n-l+1} \}
\end{align*}
contains at least two distinct sequences (\cite{MaDocMath2002}).
\begin{proposition}[{\cite[Theorem 3.6, Theorem 4.3]{MaDocMath2002}}]
Let ${\frak{L}}$ be a left-resolving $\lambda$-graph system.
The $C^*$-algebra $\OFL$
is a universal unital $C^*$-algebra
generated by
partial isometries
$S_{\alpha}$ for $\alpha \in \Sigma$
and projections
$E_i^l$ for $v_i^l \in V_l 
$
 subject to the  following relations called $({\frak{L}})$:
\begin{gather*}
\sum_{\beta \in \Sigma} S_{\beta}S_{\beta}^* 
= \sum_{i=1}^{m(l)} E_i^l   =1,  
\qquad S_\alpha S_\alpha^* E_i^l  =   E_i^{l} S_\alpha S_\alpha^* \\
S_{\alpha}^*E_i^l S_{\alpha}  =  
\sum_{j=1}^{m(l+1)} A_{l,l+1}(i,\alpha,j)E_j^{l+1},\qquad
 E_i^l  =   \sum_{j=1}^{m(l+1)}I_{l,l+1}(i,j)E_j^{l+1} 
\end{gather*}
for $\alpha \in \Sigma,$
$i=1,2,\dots,m(l),\l\in \Zp. $
\end{proposition}
 If in particular ${\frak{L}}$ satisfies condition (I),
 then any non-zero generators satisfying the above relations $({\frak{L}})$
generate an isomorphic copy of $\OFL$.
Hence $\OFL$ is a unique $C^*$-algebra subject to the relations $({\frak{L}})$
if ${\frak{L}}$ satisfies condition (I).

If $\frak{L}$ satisfies condition (I) and some irreducible condition called 
$\lambda$-irreducibility,   the $C^*$-algebra $\OFL$ is simple and purely infinite (\cite{MaMS2005}).
It is nuclear and belongs to the UCT class (\cite[Proposition 5.6]{MaDocMath2002}).
By the above relation $({\frak{L}})$, we know that 
the algebra of all finite linear combinations of the elements of the form
$$
S_{\mu}E_i^lS_{\nu}^* \quad \text{ for }
\quad \mu,\nu \in B_*(X_\Lambda), \quad i=1,\dots,m(l),\quad l \in \Zp
$$
forms a dense $*$-subalgebra of  $\OFL$.
Let us  denote by
$\DFL$ 
the $C^*$-subalgebra of $\OFL$
generated by the projections of the form
$S_\mu E_i^l S_\mu^*, i=1,\dots, m(l),\/ l \in \Zp, \, \mu\in B_*(X_\Lambda)$.
By \eqref{eq:smueilsmu}, we know that the algebra
$\DFL$ is canonically isomorphic to the commutative $C^*$-algebra 
$C(X_{\frak L})$ of continuous functions on $X_{\frak L}$.

\begin{definition}
A left-resolving $\lambda$-graph system 
${\frak L}$ is said to be {\it essentially free}\/
if the topological dynamical system
$(X_{\frak L}, \sigma_{\frak L})$ is essentially free,
which says that   
for $m\ne n$, the set 
$X_{m,n} =\{ x \in X_{\frak L} \mid 
\sigma_{\frak L}^m(x) = \sigma_{\frak L}^n(x) \}
$
does not have non empty interior
(cf. \cite[p. 19]{Renault2}).
\end{definition}
The condition is equivalent to the one that 
the \'etale groupoid
$G_{\frak L}$ is essentially principal,
so that the $C^*$-subalgebra
$\DFL$ is maximal abelian in $\OFL$.

\begin{lemma}
If a left-resolving $\lambda$-graph system 
satisfies condition (I), then it is essentially free.
\end{lemma}
\begin{proof}
Suppose that ${\frak L}$ satisfies condition (I) 
and is not essentially free.
There exist $m,n \in \Zp$ such that $m> n$ and
$X_{m,n} =\{ x \in X_{\frak L} \mid 
\sigma_{\frak L}^m(x) = \sigma_{\frak L}^n(x) \}
$
has a nonempty interior.
 By taking $l\in \N$ large enough
such as 
$l >m $, we may assume that 
$U(\mu,v_i^l) \subset X_{m,n}$
for some $\mu =(\mu_1,\dots,\mu_l) \in B_l(X_\Lambda)$ 
and
$v_i^l \in V_l.$
Let $\gamma$ be a path in ${\frak L}$ starting at a vertex in $V_m$ 
with labeled $(\mu_{m+1}, \dots,\mu_l)$ and
 ending at $v_i^l$.
Such path is unique because ${\frak L}$ is left-resolving.
Let us denote by $v_j^m\in V_m$ the source vertex of $\gamma$.
Since each sequence in $\Gamma_\infty^+(v_j^m)$ 
must be a unique periodic sequence
repeated the word $(\mu_{n+1},\cdots,\mu_m)$.
It is a contradiction to the condition (I).
\end{proof}

Let us denote by
$\mathcal{D}_\Lambda$ the $C^*$-subalgebra of $\mathcal{D}_{\frak L}$ 
generated by the projections
$S_\mu S_\mu^*, \mu \in B_*(X_\Lambda)$.
It is canonically 
isomorphic to the commutative $C^*$-algebra
$C(X_\Lambda)$ through the correspondence
$S_\mu S_\mu^* \in {\mathcal{D}}_\Lambda
\longrightarrow 
\chi_{U_\mu}\in C(X_\Lambda)$.
The inclusion
$\mathcal{D}_\Lambda \subset \mathcal{D}_{\frak L}$
is induced from the map 
$\pi_{\frak L}^*: f \in C(X_\Lambda) \longrightarrow f\circ\pi_{\frak L} \in
C(X_{\frak L}).$
We note the following lemma that is useful in our further discussions.
\begin{lemma}[{\cite[Lemma 6.5]{MaYMJ2010}}]\label{lem:masa}
Suppose that ${\frak L}$ satisfies condition (I).
Then the subalgebra $\mathcal{D}_\Lambda^\prime \cap \OFL$
of elements of $\OFL$ commuting with all elements of $\mathcal{D}_\Lambda$
coincides with $\mathcal{D}_{\frak L}$, that is,
\begin{equation*}
\mathcal{D}_\Lambda^\prime \cap \OFL = \mathcal{D}_{\frak L}. 
\end{equation*}  
Hence the position of $\mathcal{D}_\Lambda$ in $\OFL$ determines 
that of $\mathcal{D}_{\frak L}$ in $\OFL$. 
\end{lemma}

\section{Groupoid isomorphism}

In what follows, 
a left-resolving $\lambda$-graph system ${\frak L}$ is assumed to satisfy condition (I)
and hence to be essentially free.
Let
$N(\OFL,\DFL)$ be the normalizer of $\DFL$ in $\OFL$ consisting of partial isometries defined by
\begin{equation*}
N(\OFL,\DFL) : = \{
w \in \OFL
\mid
w d w^*, w^*d w \in \DFL \text{ for all } d\in \DFL
\}. 
\end{equation*}
For $w \in N(\OFL,\DFL),$
both elements $w^*w$ and $ww^*$ belong to $\DFL$.
For $\mu,\nu \in B_*(X_\Lambda), v_i^l \in V$
satisfying
$S_\mu^* S_\mu , S_\nu^* S_\nu\ge E_i^l,$
we know that
$S_\mu E_i^l S_\nu^* $ belongs to
$N(\OFL,\DFL)$ 
and satisfies
\begin{equation*}
\dom(S_\mu E_i^l S_\nu^*) = U(\nu,v_i^l),\qquad
\ran(S_\mu E_i^l S_\nu^*) = U(\mu,v_i^l)
\end{equation*}
under the natural identification between 
the algebras $\DFL$ and $C(X_{\frak L}).$ 
We will consider the Weyl groupoid
$\GNOD$
of germs of pairs
$(w, x) $ with $w \in N(\OFL,\DFL), x \in \dom(w)$
in the following way.
Two elements
$(w_1,x_1), (w_2, x_2)$ 
with
$w_i \in N(\OFL,\DFL), x_i \in \dom(w_i)$
for $i=1,2$
are said to be equivalent if
$x_1 = x_2$ and
there exists an open neighborhood $W \subset \dom(w_1)\cap \dom(w_2)$
of $x_1 = x_2$ such that 
$w_1(y) = w_2(y)$ for all $y \in W$.
The set of equivalence classes $[(w,x)]$ of the pairs $(w,x)$ is denoted by 
$\GNOD$ with partially defined product:
\begin{equation*}
[(w_1,x_1)]\cdot [(w_2, x_2)] = [(w_1w_2, x_2)]
\quad
\text{ if }
\quad w_2(x_2) = x_1
\end{equation*}
and the inverse operation
\begin{equation*}
[(w,x)]^{-1} =  [(w^*, w(x))]. 
\end{equation*}
The topology 
on $\GNOD
$
is generated by
$\{ \{[(w,x)]\mid x \in \dom(w)\}, w \in 
N(\OFL,\DFL)\}. $
It becomes an \'etale groupoid by a general theory studied by J. Renault
 (\cite[Proposition 4.10]{Renault2}, cf. \cite{Renault3}, \cite{BCW}).
 Since the groupoid $\GNOD$ is ample,
 for $[(w,x)] \in \GNOD,$
 there exists a partial isometry $v \in N(\OFL,\DFL)$
 such that $[(w,x)] = [(v,x)]$. 
 This fact was kindly informed by the referee.
For a word $\mu =(\mu_1,\dots,\mu_k) \in B_k(X_\Lambda)$,
let us denote by $|\mu|$ the length $k$.
By  \cite[Proposition 4.13]{Renault2} proved by Renault, we have 
\begin{lemma}[cf. {\cite[Proposition 4.13]{Renault2}}] \label{lem:3.1}
The correspondence
\begin{equation*}
\phi_{\frak L}:
(x, |\mu|-|\nu|, z) \in Z(U(\mu,v_i^l), |\mu|,|\nu|, U(\nu,v_i^l))
\subset
G_{\frak L}
\longrightarrow 
[(S_\mu E_i^l S_\nu^*, z)] \in \GNOD
\end{equation*}
gives rise to an isomorphism of \'etale groupoids
between $G_{\frak L}$ and $\GNOD$.
\end{lemma}

A continuous function $f:G_{\frak L} \longrightarrow \Z$ 
is called a continuous homomorphism
if it satisfies 
$
f(\gamma_1 \gamma_2) =
f(\gamma_1) + f(\gamma_2) 
$ 
for
 $\gamma_1, \gamma_2 \in G_{\frak L}$
with $\gamma_1, \gamma_2 \in G_{\frak L}.
$
It defines a one-parameter unitary
group $U_t(f), t \in \mathbb{R}/\Z = \T$
on $l^2(G_{\frak L})$
by setting
\begin{equation*}
[U_t(f)\xi](x,n,z) = \exp{(2\pi\sqrt{-1} f(x,n,z) t )}\xi(x,n,z) 
\end{equation*}
for 
$\xi \in l^2(G_{\frak L}), \, (x,n,z) \in G_{\frak L}.$
The automorphisms $\Ad(U_t(f)), t \in \T$ 
on $B(l^2(G_{\frak L}))$
leave $C^*_r(G_{\frak L})$ 
globally invariant,
and yield an action of $\T$ on $C^*_r(G_{\frak L})$.
Let us denote by $\rho^{{\frak L},f}_t$ 
the action
$\Ad(U_t(f)), t \in \T$ on $C^*_r(G_{\frak L})$.
It is called the generalized gauge action 
on $C^*_r(G_{\frak L})$.
The generalized gauge action 
$\rho^{{\frak L},c_{\frak L}}$ for the continuous function defined by
$c_{\frak L}(x,n,z) = n$ is called the gauge action on $\OFL$ and written
simply $\rho^{\frak L}$.

\begin{lemma}\label{lem:1.5}
Let
$S_\mu E_i^l S_\nu^* \in N(\OFL,\DFL)$
and
a  continuous homomorphism $f:G_{\frak L} \longrightarrow \Z$ 
be given.
For $\eta \in G_{\frak L}$, assume that $f(\eta)$ 
does not depend on $z\in U(\nu,v_i^l)$
as long as 
$\phi_{\frak L}(\eta) =[(S_\mu E_i^l S_\nu^*, z)]$.
For $m \in \Z,$
we have
\begin{equation}
\rho_t^{{\frak L},f}(S_\mu E_i^l S_\nu^*)
=\exp(2\pi{\sqrt{-1}m t})S_\mu E_i^l S_\nu^*
\quad
\text{ for all }
t \in \mathbb{T}
\label{eq:rholf} 
\end{equation}
if and only if $m = f(\eta)$.
\end{lemma}
\begin{proof}
Let us represent the $C^*$-algebra
$\OFL$ on $l^2(G_{\frak L})$
as $\OFL = C^*_r(G_{\frak L})$,
so that 
the generator
$S_\mu E_i^l S_\nu^*$
is identified with
$\chi_{U(\mu,v_i^l,\nu)} \in C_c(G_{\frak L})$.
Put
$w =S_\mu E_i^l S_\nu^*.$
For $\zeta \in l^2(G_{\frak L}), \gamma \in G_{\frak L}$,
we have
\begin{align*}
 & [\rho^{{\frak L},f}_t(w)\zeta](\gamma) \\
= &[U_t(f) w U_t(f) ^*\zeta](\gamma) \\
= &\exp({ 2\pi\sqrt{-1} f(\gamma) t}) [ w U_t(f) ^*\zeta](\gamma) \\
= &\exp({ 2\pi\sqrt{-1} f(\gamma) t}) 
\left( \sum_{\gamma' \in G_{\frak L}} w(\gamma')
  \exp({ -2\pi\sqrt{-1} f(\gamma^{\prime-1}\gamma)t })\zeta(\gamma^{\prime-1}\gamma)\right) \\
= &\exp({ 2\pi\sqrt{-1} f(\gamma) t}) 
\left( \sum_{\gamma' \in G_{\frak L}} w(\gamma')
  \exp({ -2\pi\sqrt{-1} (f(\gamma^{\prime-1})+f(\gamma))}t)\zeta(\gamma^{\prime-1}\gamma)\right) \\
= & \sum_{\gamma' \in G_{\frak L}} w(\gamma')
  \exp({ 2\pi\sqrt{-1} f(\gamma') t})\zeta(\gamma^{\prime-1}\gamma) 
=  \sum_{\gamma' \in U(\mu,v_i^l,\nu)} 
  \exp({ 2\pi\sqrt{-1} f(\gamma') t})\zeta(\gamma^{\prime-1}\gamma) \\
\intertext{and}
 & [(\exp(2\pi\sqrt{-1}m t)w \zeta](\gamma) \\
= &\exp(2\pi\sqrt{-1}m t)\sum_{\gamma'\in G_{\frak L}}
   w(\gamma')\zeta(\gamma^{\prime-1}\gamma) 
=  \sum_{\gamma' \in U(\mu,v_i^l,\nu)} 
  \exp(2\pi\sqrt{-1} m t)
   \zeta(\gamma^{\prime-1}\gamma).
\end{align*}
Hence
\eqref{eq:rholf}
holds if and only if $f(\gamma') = m$ for all $\gamma' \in U(\mu,v_i^l,\nu)$.
As the condition $\eta \in U(\mu,v_i^l,\nu)$
is equivalent to the condition that
$\phi_{\frak L}(\eta) =[(S_\mu E_i^l S_\nu^*, z)]$
for some $z \in X_{\frak L}$,
we obtain the desired assertion. 
\end{proof}
We will show the next proposition that will be used in the following sections. 
\begin{proposition}\label{prop:onesidedgroupoid}
Let ${\frak L}_1$ and 
 ${\frak L}_2$ be left-resolving $\lambda$-graph systems satisfying condition (I).
Let
$f_1:G_{{\frak L}_1}\longrightarrow \Z$
and
$f_2:G_{{\frak L}_2}\longrightarrow \Z$
 be continuous homomorphisms of groupoids. 
 Then the following are equivalent:
 \begin{enumerate}
\renewcommand{\theenumi}{\roman{enumi}}
\renewcommand{\labelenumi}{\textup{(\theenumi)}}
\item
There exists an isomorphism
$\varphi :G_{{\frak L}_1}\longrightarrow G_{{\frak L}_2}
$ of \'etale groupoids such that 
${f_1} = {f_2}\circ \varphi$.
\item
There exists an isomorphism
$\Phi:{\mathcal{O}}_{{\frak L}_1}\longrightarrow {\mathcal{O}}_{{\frak L}_2}
$ of $C^*$-algebras
such that 
$\Phi({\mathcal{D}}_{{\frak L}_1})={\mathcal{D}}_{{\frak L}_2}
$ and
$\Phi\circ \rho_t^{{\frak L}_1,f_1} =\rho_t^{{\frak L}_2,f_2}\circ \Phi, \,t \in \mathbb{T}.$ 
\end{enumerate}
\end{proposition}
\begin{proof}
(i) $\Longrightarrow$ (ii):
Suppose that 
there exists an isomorphism
$\varphi: G_{{\frak L}_1}\longrightarrow G_{{\frak L}_2}
$ of \'etale groupoids such that 
$f_1 = f_2 \circ \varphi$.
The unit space of $G_{{\frak L}_i}$ 
is the subgroupoid
$\{(x,0,x) \mid x \in X_{{\frak L}_i} \}
$
of $G_{{\frak L}_i}$
which is denoted by
$ G_{{\frak L}_i}^0$ for $i=1,2$.
The restriction of $\varphi$ to 
$ G_{{\frak L}_1}^0$
is a homeomorphism onto 
$ G_{{\frak L}_2}^0$.
Since the unit spaces 
$ G_{{\frak L}_i}^0$
are identified with 
$X_{{\frak L}_i}$ through the map
$x \in X_{{\frak L}_i} \longrightarrow (x,0,x) \in G_{{\frak L}_i}^0$
we have a homeomorphism
$h:X_{{\frak L}_1}\longrightarrow 
X_{{\frak L}_2}$
such that 
\begin{equation*}
\varphi(x,0,x) = (h(x),0,h(x)), \qquad x \in X_{{\frak L}_1}.
\end{equation*}
We define unitaries
$
V_{\varphi}:l^2(G_{{\frak L}_2})\longrightarrow l^2(G_{{\frak L}_1})
$ 
and
$
V_{\varphi^{-1}}:l^2(G_{{\frak L}_1})\longrightarrow l^2(G_{{\frak L}_2})
$ 
by setting
\begin{align*}
[V_{\varphi}\zeta](x,n,z) &
:= \zeta(\varphi(x,n,z)) \quad \text{ for } \zeta\in  l^2(G_{{\frak L}_2}),\\
[V_{\varphi^{-1}}\eta](y,m,w) &
:= \eta(\varphi^{-1}(y,m,w)) \quad \text{ for } \eta\in  l^2(G_{{\frak L}_1}).
\end{align*} 
We have $V_{\varphi}^*=V_{\varphi^{-1}}$.
Put
$\Phi = \Ad(V_{\varphi}^*)$.
It satisfies
$\Phi(C_c(G_{{\frak L}_1})) = 
C_c(G_{{\frak L}_2})$
so that 
$\Phi(C^*_r(G_{{\frak L}_1})) = 
C^*_r(G_{{\frak L}_2}).$
As
$\varphi(G_{{\frak L}_1}^0) =G_{{\frak L}_2}^0$,
we have
$\Phi(C(X_{{\frak L}_1})) = 
C(X_{{\frak L}_2}).$
We will next show that 
$\Phi\circ \rho^{{\frak L}_1,f_1}_t =\rho^{{\frak L}_2,f_2}_t\circ \Phi$ 
in the following way.
For $\zeta \in l^2(G_{{\frak L}_2}), (y,m,w) \in G_{{\frak L}_2}$ and
$a \in C_c(G_{{\frak L}_1})$, we have
\begin{align*}
& [\Phi(\Ad(U_t(f_1))(a))\zeta](y,m,w) \\
=& [U_t(f_1)aU_t(f_1)^*V \zeta](\varphi^{-1}(y,m,w)) \\
=& \exp({ 2\pi\sqrt{-1} f_1(\varphi^{-1}(y,m,w)) t})[aU_t(f_1)^*V_\varphi \zeta](\varphi^{-1}(y,m,w)) \\
=& \exp({ 2\pi\sqrt{-1} f_1(\varphi^{-1}(y,m,w)) t})
\left(\sum_{r(\gamma) = r(\varphi^{-1}(y,m,w))} a(\gamma) (U_t(f_1)^*V_\varphi \zeta)(\gamma^{-1}\varphi^{-1}(y,m,w))\right) \\
=& \exp({ 2\pi\sqrt{-1} f_1(\varphi^{-1}(y,m,w)) t})\\
&\left(\sum_{r(\gamma) = h^{-1}(y)} a(\gamma) 
\exp({-2\pi\sqrt{-1} f_1(\gamma^{-1}\varphi^{-1}(y,m,w)) t})
(V_\varphi \zeta)(\gamma^{-1}\varphi^{-1}(y,m,w))\right) \\
=& \exp({ 2\pi\sqrt{-1} f_1(\varphi^{-1}(y,m,w)) t})\\
 &\left(\sum_{r(\gamma) = h^{-1}(y)} a(\gamma) 
\exp({-2\pi\sqrt{-1} (f_1(\gamma^{-1}) + f_1(\varphi^{-1}(y,m,w)) t})
(V_\varphi \zeta)(\gamma^{-1}\varphi^{-1}(y,m,w))\right) \\
=& 
\sum_{r(\gamma) = h^{-1}(y)} a(\gamma) 
\exp({-2\pi\sqrt{-1} (f_1(\gamma^{-1})) t})
\zeta (\varphi(\gamma^{-1})(y,m,w)))
\end{align*}
and
\begin{align*}
& [\Ad(U_t(f_2))(\Phi(a))\zeta](y,m,w) \\
=& [U_t(f_2)V_\varphi^*a V_\varphi U_t(f_2)^* \zeta](y,m,w) \\
=& \exp({ 2\pi\sqrt{-1} f_2(y,m,w) t})[a V_\varphi U_t(f_2)^* \zeta](\varphi^{-1}(y,m,w))\\
=& \exp({ 2\pi\sqrt{-1} f_2(y,m,w) t})
\left(
\sum_{r(\gamma) = r(\varphi^{-1}(y,m,w))} 
a(\gamma)  V _\varphi U_t(f_2)^* \zeta(\gamma^{-1}\varphi^{-1}(y,m,w))\right) \\
=& \exp({ 2\pi\sqrt{-1} f_2(y,m,w) t})
\left(
\sum_{r(\gamma) = h^{-1}(y)} 
a(\gamma) U_t(f_2)^* \zeta (\varphi(\gamma^{-1})(y,m,w))\right) \\
=& \exp({ 2\pi\sqrt{-1} f_2(y,m,w) t})
\left(
\sum_{r(\gamma) = h^{-1}(y)} 
a(\gamma) \exp(-{ 2\pi\sqrt{-1} f_2(\varphi(\gamma^{-1})(y,m,w)) t})
\zeta (\varphi(\gamma^{-1})(y,m,w))\right) \\
=& \exp({ 2\pi\sqrt{-1} f_2(y,m,w) t}) \\
 & \left(
\sum_{r(\gamma) = h^{-1}(y)} 
a(\gamma) \exp(-{ 2\pi\sqrt{-1} (f_2(\varphi(\gamma^{-1})) + f_2(y,m,w)) t})
\zeta (\varphi(\gamma^{-1})(y,m,w))\right) \\
=& 
\sum_{r(\gamma) = h^{-1}(y)} 
a(\gamma) \exp(-{ 2\pi\sqrt{-1} f_2(\varphi(\gamma^{-1})) t})
\zeta(\varphi(\gamma^{-1})(y,m,w)). 
\end{align*}
By the assumption $f_1 = f_2\circ\varphi$,
we have
$\Phi\circ \Ad(U_t(f_1)) =\Ad(U_t(f_2))\circ \Phi$
and hence 
$\Phi\circ \rho_t^{{\frak L}_1,f_1} =\rho_t^{{\frak L}_2,f_2}\circ \Phi, \, t \in \mathbb{T}.$

(ii) $\Longrightarrow$ (i):
Suppose that 
there exists an isomorphism
$\Phi:{\mathcal{O}}_{{\frak L}_1}\longrightarrow {\mathcal{O}}_{{\frak L}_2}
$ 
of $C^*$-algebras such that 
$
\Phi({\mathcal{D}}_{{\frak L}_1})={\mathcal{D}}_{{\frak L}_2}
$ 
and
$\Phi\circ \rho_t^{{\frak L}_1,f_1} =\rho_t^{{\frak L}_2,f_2}\circ \Phi, \, t \in {\mathbb{T}}.$
We identify ${\mathcal{D}}_{{\frak L}_i}$ with
$C(X_{{\frak L}_i}).$
Since 
$
{\mathcal{O}}_{{\frak L}_i} = C^*(G_{{\frak L}_i}),
{\mathcal{D}}_{{\frak L}_i} = C^*(G_{{\frak L}_i}^0),
i=1,2,
$
Renault's result \cite[Proposition 4.11]{Renault2} ensures that
there exists an isomorphism
$\psi:G_{N({\mathcal{O}}_{{\frak L}_1},{\mathcal{D}}_{{\frak L}_1})}
\longrightarrow
G_{N({\mathcal{O}}_{{\frak L}_2},{\mathcal{D}}_{{\frak L}_2})}
$
of \'etale groupoids
and a homeomorphism
$h:X_{{\frak L}_1}\longrightarrow X_{{\frak L}_2}$
such that 
\begin{equation}
\psi([(w,x)]) = [(\Phi(w),h(x))],
\qquad
[(w,x)] \in G_{N({\mathcal{O}}_{{\frak L}_1},{\mathcal{D}}_{{\frak L}_1})}.
\label{eq:1.6.1}
\end{equation}
Let
$\phi_{{\frak L}_i}: G_{{\frak L}_i}\longrightarrow 
G_{N({\mathcal{O}}_{{\frak L}_i},{\mathcal{D}}_{{\frak L}_i})}, i=1,2
$
be the isomorphism of \'etale groupoids defined in Lemma \ref{lem:3.1}.
Define an isomorphism 
$\varphi: G_{{\frak L}_1}\longrightarrow
G_{{\frak L}_2}$
of \'etale groupoids
by
$\varphi := \phi_{{\frak L}_2}^{-1}\circ \psi\circ\phi_{{\frak L}_1}.
$
We will show that the equality
$f_1 = f_2 \circ\varphi$ holds.
For $\gamma \in G_{{\frak L}_1}$,
take
$w \in N({\mathcal{O}}_{{\frak L}_1},{\mathcal{D}}_{{\frak L}_1})$
and
$x = d(\gamma) \in X_{{\frak L}_1}$ such that 
\begin{equation}
\phi_{{\frak L}_1}(\gamma) = [(w,x)]. \label{eq:1.6.2}
\end{equation}
Since $f_1: G_{{\frak L}_1} \longrightarrow \Z$ is continuous, 
one may assume that $f_1(\gamma)$ does not depend on $x$ 
as long as $\gamma$ satisfies \eqref{eq:1.6.2}.  
By Lemma \ref{lem:1.5},
we have
\begin{equation}
\rho^{{\frak L}_1,f_1}_t(w) 
=\exp{(2\pi\sqrt{-1} f_1(\gamma) t)} w, \qquad t \in \T. \label{eq:3.4}
\end{equation} 
By \eqref{eq:1.6.1} and \eqref{eq:1.6.2},
we have
\begin{equation*}
[(\Phi(w),h(x))]
=\psi(\phi_{{\frak L}_1}(\gamma)) 
=  \phi_{{\frak L}_2}(\varphi(\gamma)).
\end{equation*}
Since both $f_2: G_{{\frak L}_2} \longrightarrow \Z$ and
$\varphi:G_{{\frak L}_1} \longrightarrow G_{{\frak L}_2}$ are continuous, 
one may assume that $f_2(\varphi(\gamma))$ does not depend on $h(x)$ 
as long as $\gamma$ satisfies 
$[(\Phi(w),h(x))]=  \phi_{{\frak L}_2}(\varphi(\gamma)).$  
By Lemma \ref{lem:1.5}, we have
\begin{equation*}
\rho_t^{{\frak L}_2,f_2}(\Phi(w)) 
=\exp{(2\pi\sqrt{-1} f_2(\varphi(\gamma)) t)}\Phi(w), \qquad t \in \T.
\end{equation*} 
Since
$\Phi\circ \rho_t^{{\frak L}_1,f_1} =\rho_t^{{\frak L}_2,f_2}\circ \Phi,$ 
we have
\begin{equation}
\rho_t^{{\frak L}_1,f_1}(w) 
=\Phi^{-1}(\exp{(2\pi\sqrt{-1} f_2(\varphi(\gamma)) t)}\Phi(w))
=\exp{(2\pi\sqrt{-1} f_2(\varphi(\gamma)) t)}w, \quad t \in \T 
\label{eq:1.6.3}
\end{equation} 
so that by \eqref{eq:3.4} and \eqref{eq:1.6.3} we obtain
$
f_1(\gamma) = f_2(\varphi(\gamma))
$ for 
$\gamma \in G_{{\frak L}_1}$ 
and hence 
$f_1 = f_2 \circ \varphi$.
\end{proof}

\section{Continuous orbit equivalence on one-sided subshifts}
Let $\frak L$ be a left-resolving $\lambda$-graph system 
over $\Sigma$ satisfying condition (I).
Recall that  $X_\Lambda$ 
denotes the shift space 
\begin{equation*}
X_\Lambda =\{(\alpha_i)_{i \in \N}\in \Sigma^{\N}
 \mid (\alpha_i,u_i)_{i\in\N}\in X_{\frak L} \text{ for some } u_i\in \Omega_{\frak L}, i\in \N\}  
\end{equation*}
of the right one-sided subshift
$(X_\Lambda,\sigma_\Lambda)$
with the shift transformation
$\sigma_{\Lambda}((\alpha_i)_{i \in \N}) =(\alpha_{i+1})_{i \in \N} 
$
for
$(\alpha_i)_{i \in \N} \in X_\Lambda$.
Then the correspondence
$\pi_{\frak L}: (\alpha_i,u_i)_{i\in \N} \in X_{\frak L}
\longrightarrow
(\alpha_i)_{i\in \N}\in X_\Lambda
$
gives rise to a continuous surjection
$\pi_{\frak L}: X_{\frak L} \longrightarrow X_\Lambda$
such that 
$\pi_{\frak L}\circ \sigma_{\frak L} = \sigma_{\Lambda}\circ\pi_{\frak L}$.

For two left-resolving $\lambda$-graph systems ${\frak L}_1$ and 
 ${\frak L}_2$,
denote by 
$(X_{\Lambda_1},\sigma_{\Lambda_1})$ and
$(X_{\Lambda_2},\sigma_{\Lambda_2})$ 
the right one-sided subshifts presented by
 ${\frak L}_1$ and 
 ${\frak L}_2$, respectively.

\begin{definition}[{\cite[Section 6]{MaYMJ2010}}] \label{def:lambdacoe}
 One-sided subshifts 
$(X_{\Lambda_1},\sigma_{\Lambda_1})$
and
$(X_{\Lambda_2},\sigma_{\Lambda_2})$
are said to be
$({\frak L}_1,{\frak L}_2)$-{\it continuously orbit equivalent}\/ if
there exist homeomorphisms
$h_{\frak L}: X_{{\frak L}_1}\longrightarrow X_{{\frak L}_2}$
and
$h_{\Lambda}: X_{\Lambda_1}\longrightarrow X_{\Lambda_2}$
satisfying 
$\pi_{{\frak L}_2} \circ h_{\frak L} =h_{\Lambda}\circ\pi_{{\frak L}_1}$
and  continuous maps
$k_i, l_i : X_{{\frak L}_i} \longrightarrow \Zp, i=1,2$ satisfying 
\begin{align}
\sigma_{{\frak L}_2}^{k_1(x)}(h_{\frak L}(\sigma_{{\frak L}_1}(x)))
& = 
\sigma_{{\frak L}_2}^{l_1(x)}(h_{\frak L}(x)),\qquad x \in X_{{\frak L}_1}, \label{eq:coelk1} \\ 
\sigma_{{\frak L}_1}^{k_2(y)}(h_{\frak L}^{-1}(\sigma_{{\frak L}_2}(y)))
& = 
\sigma_{{\frak L}_1}^{l_2(y)}(h_{\frak L}^{-1}(y)),\qquad y \in X_{{\frak L}_2}. \label{eq:coelk2}
\end{align}
\end{definition}
The above situation was called
that the factor maps
$\pi_{{\frak L}_1}:X_{{\frak L}_1}\longrightarrow X_{\Lambda_1}$ and
$\pi_{{\frak L}_2}:X_{{\frak L}_2}\longrightarrow X_{\Lambda_2}$ 
are continuously orbit equivalent in \cite{MaYMJ2010}.

\begin{proposition} \label{prop:coe}
The following are equivalent.
\begin{enumerate}
\renewcommand{\theenumi}{\roman{enumi}}
\renewcommand{\labelenumi}{\textup{(\theenumi)}}
\item One-sided subshifts 
$(X_{\Lambda_1},\sigma_{\Lambda_1})$
and
$(X_{\Lambda_2},\sigma_{\Lambda_2})$
are 
$({\frak L}_1,{\frak L}_2)$-continuously orbit equivalent.
\item
There exist an isomorphism
$\varphi:G_{{\frak L}_1}\longrightarrow G_{{\frak L}_2}$
of \'etale groupoids and a homeomorphism
$h:X_{\Lambda_1}\longrightarrow X_{\Lambda_2}$
such that
$\pi_{{\frak L}_2} \circ\varphi|_{X_{{\frak L}_1}} = h\circ\pi_{{\frak L}_1}$,
where $X_{{\frak L}_i}$ is identified with
the unit space $G_{{\frak L}_i}^0$ of the \'etale groupoid $G_{{\frak L}_i}, i=1,2.$
\end{enumerate}
\end{proposition}
\begin{proof}
(ii) $\Longrightarrow$ (i):
Suppose that  
there exist an isomorphism
$\varphi:G_{{\frak L}_1}\longrightarrow G_{{\frak L}_2}$
of \'etale groupoids and a homeomorphism
$h:X_{\Lambda_1}\longrightarrow X_{\Lambda_2}$
such that
$\pi_{{\frak L}_2} \circ\varphi|_{X_{{\frak L}_1}} = h\circ\pi_{{\frak L}_1}.$
As the shift spaces
 $X_{{\frak L}_i}$ are identified with
the unit spaces $G_{{\frak L}_i}^0$ of the groupoids $G_{{\frak L}_i}$
for $i=1,2$,
the  isomorphism
$\varphi:G_{{\frak L}_1}\longrightarrow G_{{\frak L}_2}$
of \'etale groupoids
yields a homeomorphism from
$X_{{\frak L}_1}$ onto $X_{{\frak L}_2}$,
which we denote by $h_{\frak L}$.
It satisfies
$\varphi(x,0,x) = (h_{\frak L}(x),0,h_{\frak L}(x))$ for $x \in X_{\frak L}$.
Since $(x,1,\sigma_{{\frak L}_1}(x))\in G_{{\frak L}_1}$
for $x \in X_{{\frak L}_1}$, we have
$\varphi(x,1,\sigma_{{\frak L}_1}(x)) \in G_{{\frak L}_2}$
and 
\begin{align*}
\varphi(x,1,\sigma_{{\frak L}_1}(x))
& = \varphi(x,0,x)
      \varphi(x,1,\sigma_{{\frak L}_1}(x)) 
      \varphi(\sigma_{{\frak L}_1}(x), 0, \sigma_{{\frak L}_1}(x)) \\
& =  (h_{\frak L}(x),0,h_{\frak L}(x)) \varphi(x,1,\sigma_{{\frak L}_1}(x)) 
      (h_{\frak L}(\sigma_{{\frak L}_1}(x)),0,h_{\frak L}(\sigma_{{\frak L}_1}(x))),
\end{align*}
one may find an integer $n_2(x) \in\Z$ such that 
\begin{equation}
\varphi(x,1,\sigma_{{\frak L}_1}(x))
=  (h_{\frak L}(x), n_2(x),h_{\frak L}(\sigma_{{\frak L}_1}(x))). \label{eq:n2}
\end{equation}
For $\alpha \in \Sigma, v_i^1 \in V_1$ 
we put a clopen set in $X_{{\frak L}_1}$
\begin{equation*}
U(\alpha,v_i^1) =
\{ (\alpha_n,u_n)_{n\in\N}\in X_{{\frak L}_1}\mid
\alpha_1 = \alpha, u^1_1 = v_i^1\},
\end{equation*}
whose characteristic function
$\chi_{U(\alpha,v_i^l)}$ is regarded as the projection
$S_\alpha E_i^1 S_\alpha^*$ in $\mathcal{O}_{{\frak L}_1}.$
We then have
\begin{equation*}
X_{{\frak L}_1} = \bigcup_{\alpha \in \Sigma} \bigcup_{i=1}^{m(1)}U(\alpha,v_i^1).
\end{equation*}
We put
\begin{equation*}
V(\alpha,v_i^1) =\sigma_{{\frak L}_1}(U(\alpha,v_i^1))
( = \{ (\alpha_n,u_n)_{n\in\N}\in X_{\frak L}\mid
A(v_i^1, \alpha_1, u^2_1) =1 \}),
\end{equation*}
where $u_1 =(u_1^l)_{l\in \N} \in \Omega_{\frak L}.$
Since 
\begin{equation*}
V(\alpha,v_i^1) =\bigcup_{A(v_i^1, \alpha_1, u_1^2) =1}U(\alpha_1,u^2_1),
\end{equation*}
the set 
$V(\alpha,v_i^1)$ is clopen in $X_{{\frak L}_1}.$
Hence  the set
\begin{equation*}
Z_1(U(\alpha, v_i^1), 1,0, V(\alpha,v_i^1))
 =\{
(x,1,z) \in G_{{\frak L}_1} \mid x \in U(\alpha,v_i^1), z \in V(\alpha, v_i^1), 
\sigma_{\frak L}(x) =z\}
\end{equation*}
is an open and compact subset of $G_{{\frak L}_1}$,
and so is 
$
\varphi(Z(U(\alpha, v_i^1), 1,0, V(\alpha,v_i^1)))
$
 in
$G_{{\frak L}_2}$.
Hence
\begin{equation}
\varphi(Z_1(U(\alpha, v_i^1), 1,0, V(\alpha,v_i^1)))
 =\bigcup_{j=1}^n
Z_2(U_j,l_j, k_j,V_j) \label{eq:varphiUV}
\end{equation}
for some clopen sets
$U_j, V_j \subset X_{{\frak L}_2}$
and $l_j,k_j \in \Zp$. 
Define
$k_{(\alpha,v_i^1)}(x) = k_j $
and
$l_{(\alpha,v_i^1)}(x) = l_j $
 for $x \in h_{\frak L}^{-1}(U_j)$.
The both functions
$k_{(\alpha,v_i^1)} $
and
$l_{(\alpha,v_i^1)}$
are continuous on
$U(\alpha,v_i^1)$.
For $x \in U(\alpha,v_i^1)$,
one has
$h_{\frak L}(x) \in U_j$ for some $j$ 
and
\begin{equation}
\varphi(x,1,\sigma_{{\frak L}_1}(x)) =
(h_{\frak L}(x), l_j-k_j, h_{\frak L}(\sigma_{{\frak L}_1}(x)))
\in Z_2(U_j,l_j, k_j,V_j) \label{eq:4.7}
\end{equation}
so that 
\begin{equation}
\sigma_{{\frak L}_2}^{k_j}(h_{\frak L}( \sigma_{{\frak L}_1}(x)))
= \sigma_{{\frak L}_2}^{l_j}(h_{\frak L}(x)),
\qquad 
x \in U(\alpha,v_i^1). \label{eq:4.8}
\end{equation}
Since 
$
\bigcup_{\alpha, v_i^l}\bigcup_{j=1}^n h_{\frak L}^{-1}(U_j) 
= X_{{\frak L}_1},
$
we have continuous maps
$k_1, l_1: X_{{\frak L}_1}\longrightarrow \Z$
such that 
\begin{equation*}
\sigma_{{\frak L}_2}^{k_1(x)}(h_{\frak L}( \sigma_{{\frak L}_1}(x)))
= \sigma_{{\frak L}_2}^{l_1(x)}(h_{\frak L}(x)),
\qquad 
x \in X_{{\frak L}_1}.
\end{equation*}
We know that there exist
continuous maps
$k_2, l_2:X_{{\frak L}_2}\longrightarrow \Z$ 
such that 
\begin{equation*}
\sigma_{{\frak L}_1}^{k_2(y)}(h_{\frak L}^{-1}(\sigma_{{\frak L}_2}(y)))
= 
\sigma_{{\frak L}_1}^{l_2(y)}(h_{\frak L}^{-1}(y)), \qquad y \in X_{{\frak L}_2}
\end{equation*}
in a similar way.
Therefore 
the one-sided subshifts 
$(X_{\Lambda_1},\sigma_{\Lambda_1})$
and
$(X_{\Lambda_2},\sigma_{\Lambda_2})$
are 
$({\frak L}_1,{\frak L}_2)$-continuously orbit equivalent.

(i) $\Longrightarrow$ (ii):
For a function $f:X_{\frak L}\longrightarrow \Z$ and $n \in \N$,
we set
$f^n(x) = \sum_{i=0}^{n-1}f(\sigma_{\frak L}^i(x))$ for $x \in X_{\frak L}.$
Suppose that the one-sided subshifts 
$(X_{\Lambda_1},\sigma_{\Lambda_1})$
and
$(X_{\Lambda_2},\sigma_{\Lambda_2})$
are 
$({\frak L}_1,{\frak L}_2)$-continuously orbit equivalent.
There exist
 homeomorphisms
$h_{\frak L}: X_{{\frak L}_1}\longrightarrow X_{{\frak L}_2}$
and
$h_{\Lambda}: X_{\Lambda_1}\longrightarrow X_{\Lambda_2}$
satisfying the equalities
$\pi_{{\frak L}_2}\circ h_{\frak L} = h_\Lambda\circ\pi_{{\frak L}_1}$
and
\eqref{eq:coelk1}, \eqref{eq:coelk2}. 
We will show that 
for $(x,n,z) \in G_{{\frak L}_1}$, 
there exists a continuous function
$n_1: G_{{\frak L}_1}\longrightarrow \Z$
such that 
$(h_{\frak L}(x), n_1(x,n,z), h_{\frak L}(z)) \in G_{{\frak L}_2}$.
For $(x,n,z) \in G_{{\frak L}_1}$,
there exist $p,q \in \Zp$ such that
$p-q =n$ and
$\sigma_{{\frak L}_1}^p(x) =\sigma_{{\frak L}_1}^q(z). $
By the equality \eqref{eq:coelk1}, we have
\begin{equation*}
\sigma_{{\frak L}_2}^{k_1^p(x)}(h_{\frak L}(\sigma_{{\frak L}_1}^p(x)))
 =  
\sigma_{{\frak L}_2}^{l_1^p(x)}(h_{\frak L}(x)) \quad
\text{ and }
\quad
\sigma_{{\frak L}_2}^{k_1^q(z)}(h_{\frak L}(\sigma_{{\frak L}_1}^q(z)))
 =  
\sigma_{{\frak L}_2}^{l_1^q(z)}(h_{\frak L}(z)).
\end{equation*}
Hence the equality
\begin{equation*}
\sigma_{{\frak L}_2}^{l_1^p(x)+k_1^q(z)}(h_{\frak L}(x)) 
 = 
\sigma_{{\frak L}_2}^{l_1^q(z)+k_1^p(x)}(h_{\frak L}(z))
\end{equation*}
holds.
This implies that
$(h_{\frak L}(x), 
(l_1^p(x)+k_1^q(z))- (l_1^q(z)+k_1^p(x)), 
h_{\frak L}(z)) $
gives rise to an element of  $G_{{\frak L}_2}$.
Put
$c_1(x) = l_1(x) - k_1(x)$ and hence
$c_1^p(x) =l_1^p(x) -k_1^p(x)$
and
$c_1^q(z) =l_1^q(z) -k_1^q(z).$
Define
the function $n_1: G_{{\frak L}_1}\longrightarrow \Z$ 
 by
$n_1(x,p-q, z) =(l_1^p(x)+k_1^q(z))- (l_1^q(z)+k_1^p(x)) =c_1^p(x) -c_1^q(z).$
It is routine to check that 
$n_1(x,p-q, z) $ is well-defined,
so that 
\begin{equation*}
(h_{\frak L}(x), c_1^p(x) -c_1^q(z), 
h_{\frak L}(z)) \in G_{{\frak L}_2}.
\end{equation*} 
It is not difficult to see that
the correspondence
\begin{equation}
(x,p-q,z) \in G_{{\frak L}_1} \longrightarrow
(h_{\frak L}(x), c_1^p(x) -c_1^q(z), 
h_{\frak L}(z)) \in G_{{\frak L}_2}. \label{eq:4.9}
\end{equation} 
defines a homomorphism 
$\varphi: G_{{\frak L}_1}\longrightarrow  G_{{\frak L}_2}$
of groupoids.

We set $c_2:X_{{\frak L}_2}\longrightarrow \Z$ 
by $c_2(y) = l_2(y) - k_2(y)$ for $y \in X_{{\frak L}_2}$.
By the equality \eqref{eq:coelk2}, 
we obtain a homomorphism
\begin{equation*}
(y,p-q,w) \in G_{{\frak L}_1} \longrightarrow
(h_{\frak L}^{-1}(y), c_2^p(y) -c_2^q(w), 
h_{\frak L}^{-1}(w)) \in G_{{\frak L}_1}
\end{equation*} 
in a similar way to \eqref{eq:4.9},
which is the inverse of $\varphi$.
%
Since
both functions
$c_1:X_{{\frak L}_1}\longrightarrow \Z$
$c_2:X_{{\frak L}_2}\longrightarrow \Z$
are continuous,
we have the isomorphism
$\varphi:G_{{\frak L}_1}\longrightarrow G_{{\frak L}_2}$
of \'etale groupoids,
which satisfies 
$\pi_{{\frak L}_2} \circ\varphi|_{X_{{\frak L}_1}} = h_\Lambda\circ\pi_{{\frak L}_1}$.
Therefore we get the assertion (ii).
\end{proof}


By Proposition \ref{prop:onesidedgroupoid} and Proposition \ref{prop:coe}, 
we may prove Theorem \ref{thm:onesidedcoe} in the following way.

\medskip

\noindent
{\it Proof of Theorem \ref{thm:onesidedcoe}.}
(i) $\Longleftrightarrow$ (ii) follows from Proposition \ref{prop:coe}.

(ii) $\Longrightarrow$ (iii):
Let
$\varphi:G_{{\frak L}_1}\longrightarrow G_{{\frak L}_2}$
be the isomorphism of \'etale groupoids and 
$h:X_{\Lambda_1}\longrightarrow X_{\Lambda_2}$
the  homeomorphism
in the statement (ii).
Take a continuous function $c_{{\frak L}_2}:G_{{\frak L}_2} \longrightarrow \Z$
as $f_2$ in Proposition \ref{prop:onesidedgroupoid},
and put $f_1 = c_{{\frak L}_2}\circ \varphi$.
By Proposition \ref{prop:onesidedgroupoid} (i) $\Longrightarrow$ (ii),
there exists an isomorphism
$\Phi:{\mathcal{O}}_{{\frak L}_1}\longrightarrow {\mathcal{O}}_{{\frak L}_2}
$ of $C^*$-algebras
such that 
$\Phi({\mathcal{D}}_{{\frak L}_1})={\mathcal{D}}_{{\frak L}_2}.
$ 
By the construction of $\Phi$ in the proof of Proposition \ref{prop:onesidedgroupoid} (i) $\Longrightarrow$ (ii),
one knows that 
$\Phi(f) = f\circ\varphi^{-1}$ for $f \in \mathcal{D}_{{\frak L}_1}$.
The  homeomorphism
$h:X_{\Lambda_1}\longrightarrow X_{\Lambda_2}$
satisfying 
$\pi_{{\frak L}_2} \circ\varphi|_{X_{{\frak L}_1}} = h\circ\pi_{{\frak L}_1}$
guarantees 
the condition 
$\Phi(\mathcal{D}_{\Lambda_1}) = \mathcal{D}_{\Lambda_2}.
$

(iii) $\Longrightarrow$ (ii):
Assume (iii).
By Lemma \ref{lem:masa}, the isomorphism
$\Phi:{\mathcal{O}}_{{\frak L}_1}\longrightarrow {\mathcal{O}}_{{\frak L}_2}$
such that 
$
\Phi({\mathcal{D}}_{\Lambda_1})={\mathcal{D}}_{\Lambda_2}
$ 
satisfies
$\Phi({\mathcal{D}}_{{\frak L}_1})={\mathcal{D}}_{{\frak L}_2}.
$ 
As in the proof of (ii) $\Longrightarrow$ (i) of Proposition \ref{prop:onesidedgroupoid},
put
 $\varphi = \phi_{{\frak L}_2}^{-1}\circ\Phi\circ\phi_{{\frak L}_1}:
G_{{\frak L}_1}\longrightarrow G_{{\frak L}_2}$
and 
$f_1 = c_{{\frak L}_2}\circ\varphi$.
Hence
$\varphi:G_{{\frak L}_1}\longrightarrow G_{{\frak L}_2}$
yields an isomorphism of \'etale groupoids by 
(ii) $\Longrightarrow$ (i) of Proposition \ref{prop:onesidedgroupoid}.
It is easy to see that the condition
$\Phi(\mathcal{D}_{\Lambda_1}) =
\mathcal{D}_{\Lambda_2}$
gives rise to a homeomorphism
$h:X_{\Lambda_1}\longrightarrow X_{\Lambda_2}$
such that 
$\pi_{{\frak L}_2}\circ\varphi|_{G_{{\frak L}_1}^0} = h\circ\pi_{{\frak L}_1}.$
\qed

We remark that the equivalence between (i) and (iii) in Theorem \ref{thm:onesidedcoe}
was already seen in \cite[Theorem 6.6]{MaYMJ2010}.
We gave its proof as above in this paper  for the sake of completeness.
This fact was informed the author by the referee.

\section{Eventual conjugacy of one-sided subshifts}
In this section, we will study eventual conjugacy of one-sided subshifts.
\begin{definition} \label{def:lambdaevent}
One-sided subshifts 
$(X_{\Lambda_1},\sigma_{\Lambda_1})$
and
$(X_{\Lambda_2},\sigma_{\Lambda_2})$
are said to be
$({\frak L}_1,{\frak L}_2)$-{\it eventually conjugate}\/ if
there exist homeomorphisms
$h_{\frak L}: X_{{\frak L}_1}\longrightarrow X_{{\frak L}_2}$
and
$h_{\Lambda}: X_{\Lambda_1}\longrightarrow X_{\Lambda_2}$
satisfying 
$\pi_{{\frak L}_2} \circ h_{\frak L} =h_{\Lambda}\circ\pi_{{\frak L}_1}$
and  continuous maps
$k_i : X_{{\frak L}_i} \longrightarrow \Zp, i=1,2$ 
satisfying
\begin{align}
\sigma_{{\frak L}_2}^{k_1(x)}(h_{\frak L}(\sigma_{{\frak L}_1}(x)))
& = 
\sigma_{{\frak L}_2}^{k_1(x) +1}(h_{\frak L}(x)),\qquad x \in X_{{\frak L}_1}, 
\label{eq:eventconj1} \\ 
\sigma_{{\frak L}_1}^{k_2(y)}(h_{\frak L}^{-1}(\sigma_{{\frak L}_2}(y)))
& = 
\sigma_{{\frak L}_1}^{k_2(y) +1}(h_{\frak L}^{-1}(y)),\qquad y \in X_{{\frak L}_2}. 
\label{eq:eventconj2}
\end{align}
We note that if $(X_{\Lambda_1},\sigma_{\Lambda_1})$
and
$(X_{\Lambda_2},\sigma_{\Lambda_2})$
are 
$({\frak L}_1,{\frak L}_2)$-eventually conjugate, 
then we may take the above 
functions $k_1,k_2$ as constant numbers because 
both $X_{{\frak L}_1}$ and $X_{{\frak L}_2}$ are compact.
\end{definition}
Concerning eventual conjugacy, we provide the following proposition.
Recall that the continuous homomorphism
$c_{\frak L}:G_{\frak L} \longrightarrow
\Z$ is defined by 
$c_{\frak L}(x,n,z) =n$.
\begin{proposition} \label{prop:eventconj}
The following are equivalent.
\begin{enumerate}
\renewcommand{\theenumi}{\roman{enumi}}
\renewcommand{\labelenumi}{\textup{(\theenumi)}}
\item One-sided subshifts 
$(X_{\Lambda_1},\sigma_{\Lambda_1})$
and
$(X_{\Lambda_2},\sigma_{\Lambda_2})$
are 
$({\frak L}_1,{\frak L}_2)$-eventually conjugate.
\item
There exist an isomorphism
$\varphi:G_{{\frak L}_1}\longrightarrow G_{{\frak L}_2}$
of \'etale groupoids and a homeomorphism
$h:X_{\Lambda_1}\longrightarrow X_{\Lambda_2}$
such that
$\pi_{{\frak L}_2} \circ\varphi|_{X_{{\frak L}_1}} = h\circ\pi_{{\frak L}_1}$,
and
$c_{{\frak L}_2} \circ \varphi = c_{{\frak L}_1}$
where $X_{{\frak L}_i}$ 
is identified with
the unit space $G_{{\frak L}_i}^0$ 
of the groupoid $G_{{\frak L}_i}, i=1,2.$
\end{enumerate}
\end{proposition}
\begin{proof}
(ii) $\Longrightarrow$ (i):
We follow the proof of 
 (ii) $\Longrightarrow$ (i) of Proposition \ref{prop:coe}
and further assume that $c_{{\frak L}_2}\circ\varphi = c_{{\frak L}_1}$.
In the proof, we see that 
the equality \eqref{eq:n2} holds.
Since
\begin{equation*}
c_{{\frak L}_2}(\varphi(x,1,\sigma_{{\frak L}_1}(x)))
= 
c_{{\frak L}_2}(h_{\frak L}(x), l_j-k_j, h_{\frak L}(\sigma_{{\frak L}_1}(x)))
= l_j -k_j
\end{equation*}
and
$c_{{\frak L}_1}(x, 1,\sigma_{{\frak L}_1}(x)) =1$,
we have
$l_j -k_j =1$ for all $x \in h_{\frak L}^{-1}(U_j)$, so that
one may take
$l_1(x) = k_1(x) +1$ for all $x \in X_{{\frak L}_1}$. 
We know that $l_2(y) = k_2(y) +1$ in a similar way.
Thus 
$(X_{\Lambda_1},\sigma_{\Lambda_1})$
and
$(X_{\Lambda_2},\sigma_{\Lambda_2})$
are $({\frak L}_1,{\frak L}_2)$-eventually conjugate.

(i) $\Longrightarrow$ (ii):
We follow the proof of (i) $\Longrightarrow$ (ii) of Proposition \ref{prop:coe},
and  further assume that 
$(X_{\Lambda_1},\sigma_{\Lambda_1})$
and
$(X_{\Lambda_2},\sigma_{\Lambda_2})$
are 
$({\frak L}_1,{\frak L}_2)$-eventually conjugate.
We may assume that there exists a constant number 
$K \in \Zp$ such that 
\begin{align*}
\sigma_{{\frak L}_2}^{K}(h_{\frak L}(\sigma_{{\frak L}_1}(x)))
& = 
\sigma_{{\frak L}_2}^{K +1}(h_{\frak L}(x)),\qquad x \in X_{{\frak L}_1}, 
\\ 
\sigma_{{\frak L}_1}^{K}(h_{\frak L}^{-1}(\sigma_{{\frak L}_2}(y)))
& = 
\sigma_{{\frak L}_1}^{K +1}(h_{\frak L}^{-1}(y)), \qquad y \in X_{{\frak L}_2}. 
\end{align*}
In the proof of  (i) $\Longrightarrow$ (ii) of Proposition \ref{prop:coe},
we may take $k_1(x) =K, l_1(x)  =K+1$ for all $x \in X_{{\frak L}_1}$.
 It then follows that
\begin{equation*}
k_1^p(x) =\sum_{n=0}^{p-1}k_1(\sigma_{{\frak L}_1}(x)) = p\cdot K
\end{equation*}
and 
$l_1^p(x) = p\cdot (K+1)$
similarly.
Hence we have
$c_1^p(x) = p, c_1^q(z) =q$ so that 
\begin{equation*}
\varphi(x,p-q,z) = (h_{\frak L}(x), p-q, h_{\frak L}(z)).
\end{equation*}
Therefore
\begin{equation*}
(c_{{\frak L}_2}\circ\varphi)(x,p-q,z) 
= c_{{\frak L}_2}(h_{\frak L}(x), p-q, h_{\frak L}(z))
=p-q 
=c_{{\frak L}_1}(x,p-q,z),
\end{equation*}
and hence
$c_{{\frak L}_2}\circ\varphi = c_{{\frak L}_1}$.
\end{proof}

By Proposition \ref{prop:onesidedgroupoid} 
and Proposition \ref{prop:eventconj}, 
we give a proof of Theorem \ref{thm:eventconj}.

\medskip

\noindent
{\it Proof of Theorem \ref{thm:eventconj}.}
The equivalence between (i) and (ii) follows from Proposition \ref{prop:eventconj}.

Assume the condition (ii).
By Proposition \ref{prop:onesidedgroupoid} (i) $\Longrightarrow$ (ii),
there exists an isomorphism
$\Phi:{\mathcal{O}}_{{\frak L}_1}\longrightarrow {\mathcal{O}}_{{\frak L}_2}
$ 
of $C^*$-algebras
such that 
$
\Phi({\mathcal{D}}_{{\frak L}_1})={\mathcal{D}}_{{\frak L}_2}
$
and
$
\Phi \circ \rho_t^{{\frak L}_1} =\rho_t^{{\frak L}_2}\circ \Phi, \, t \in \T.
$
The  isomorphism
$\Phi:{\mathcal{O}}_{{\frak L}_1}\longrightarrow {\mathcal{O}}_{{\frak L}_2}
$ defined
in the proof of 
Proposition \ref{prop:onesidedgroupoid} (i) $\Longrightarrow$ (ii)
satisfies 
$\Phi(f) = f \circ \varphi^{-1}$
for
$f \in C_c(G_{\frak L}).$
Hence the additional condition  
$\pi_{{\frak L}_2} \circ\varphi|_{X_{{\frak L}_1}} = h\circ\pi_{{\frak L}_1}$
ensures the condition
$
\Phi({\mathcal{D}}_{{\Lambda}_1})={\mathcal{D}}_{{\Lambda}_2},
$
showing the assertion (iii).

Conversely,
assume the condition (iii).
By Lemma \ref{lem:masa}, the isomorphism
$\Phi:{\mathcal{O}}_{{\frak L}_1}\longrightarrow {\mathcal{O}}_{{\frak L}_2}$
such that 
$
\Phi({\mathcal{D}}_{\Lambda_1})={\mathcal{D}}_{\Lambda_2}
$ 
satisfies
$\Phi({\mathcal{D}}_{{\frak L}_1})={\mathcal{D}}_{{\frak L}_2}.
$ 
By Proposition \ref{prop:onesidedgroupoid} (ii) $\Longrightarrow$ (i),
there exists an isomorphism
$\varphi:G_{{\frak L}_1}\longrightarrow G_{{\frak L}_2}$
of \'etale groupoids
such that 
$c_{{\frak L}_2} \circ \varphi = c_{{\frak L}_1}$. 
Since 
$\Phi({\mathcal{D}}_{{\Lambda}_1})={\mathcal{D}}_{{\Lambda}_2}$
and the algebras ${\mathcal{D}}_{{\Lambda}_i}$
are naturally identified with
$C(X_{{\Lambda}_i}), i=1,2$ respectively,
There exists a 
homeomorphism
$h:X_{{\frak L}_1}\longrightarrow X_{{\frak L}_2}$
such that 
$\Phi(f) = f \circ h^{-1}$ 
for 
$f\in C(X_{{\frak L}_1}).$
As 
$\Phi({\mathcal{D}}_{{\frak L}_1})={\mathcal{D}}_{{\frak L}_2}$
and the inclusion
${\mathcal{D}}_{{\Lambda}_i} \hookrightarrow {\mathcal{D}}_{{\frak L}_i}$
is induced by 
the factor map
$\pi_{{\frak L}_i}: X_{{\frak L}_i} \longrightarrow X_{\Lambda_i}, i=1,2$
we know that 
$\pi_{{\frak L}_2} \circ\varphi|_{X_{{\frak L}_1}} = h\circ\pi_{{\frak L}_1}.$
Hence we get the assertion (ii).
\qed

\section{Stabilization of $\lambda$-graph systems}

For a left-resolving $\lambda$-graph system
${\frak L} =(V,E,\lambda,\iota)$ over $\Sigma$,
we will construct
its stabilization
$\widetilde{\frak L} =(\widetilde{V},\widetilde{E}, \tilde{\lambda},\tilde{\iota})$
and the associated groupoid
$G_{\widetilde{\frak L}}$
such that 
its groupoid $C^*$-algebra
$C^*(G_{\widetilde{\frak L}})$ which is written
${\mathcal{O}}_{\widetilde{\frak L}}$, 
its canonical maximal abelian $C^*$-subalgebra
${\mathcal{D}}_{\widetilde{\frak L}}$,
and its gauge action
$\rho^{\widetilde{\frak L}}$
are isomorphic to
$\OFL\otimes\K$,
$\DFL\otimes\C$,
and
$\rho^{\frak L}\otimes\id$,
respectively.
The idea of the construction of 
$\widetilde{\frak L}$ is essentially due to the Tomforde's construction 
for directed graphs in 
\cite{Tomforde2}.
For each vertex $v_i^l \in V_l$ in $\frak L$,
we set
$v_i^l(0) =v_i^l$ and
attach a finite sequence of edges 
$e_{k,v_i^l}, k=1,2,\dots,p$ and
vertices
$v_i^l(k), k=1,2,\dots,p$ such as
\begin{equation*}
\begin{CD}
v_i^l(p) \; @>e_{p,v_i^l}>>\;  
v_i^l(p-1)\; @>e_{p-1,v_i^l}>>\; \cdots 
@>e_{3,v_i^l}>>\;
v_i^l(2)\; @>e_{2,v_i^l}>>\;  
v_i^l(1)\; @>e_{1,v_i^l}>>\; v_i^l(0) =v_i^l  
\end{CD}
\end{equation*}
Hence
the new edges and new vertices satisfy
\begin{equation*}
s(e_{n,v_i^l}) = v_i^l(n),\quad
t(e_{n,v_i^l}) = v_i^l(n-1) \quad \text{ for }
\qquad n=1,2,\dots,p.
\end{equation*}
 They are called a head of the vertex $v_i^l$.
For an 
$\iota$-orbit
$u = (u^n)_{n\in \Zp} \in \Omega_{\frak L}$,
we may consider a sequence of heads
\begin{equation*}
\begin{CD}
u^n(p) \; @>e_{p,u^n}>>\;  
u^n(p-1)\; @>e_{p-1,u^n}>>\; \cdots 
@>e_{3,u^n}>>\;
u^n(2)\; @>e_{2,u^n}>>\;  
u^n(1)\; @>e_{1,u^n}>>\; u^n(0) =u^n,  
\end{CD}
\end{equation*}
and  extend the $\iota$-map
to the heads by setting
\begin{equation*}
\tilde{\iota}(u^{n+1}(k)) =u^{n}(k), \qquad k=1,\dots,p.
\end{equation*}
Then  a head of the $\iota$-orbit
$u = (u^n)_{n\in \Zp} \in \Omega_{\frak L}$
is defined and written 
\begin{equation*}
\begin{CD}
u(p) \; @>e_{p,u}>>\;  
u(p-1)\; @>e_{p-1,u}>>\; \cdots 
@>e_{3,u}>>\;
u(2)\; @>e_{2,u}>>\;  
u(1)\; @>e_{1,u}>>\; u(0) =u.  
\end{CD}
\end{equation*}
For $x = (\alpha_i, u_i)_{i\in \N} \in X_{\frak L}$
so that 
$(u_i,\alpha_{i+1},u_{i+1}) \in E_{\frak L}$
and
$(u_0,\alpha_{1},u_{1}) \in E_{\frak L}$
for some $u_0\in \Omega_{\frak L}$,
and $p \in \Zp$,
define $e_{[p]}x$ by 
\begin{equation*}
\begin{CD}
u_0(p) \; @>e_{p,u_0}>>\;  
u_0(p-1)\; @>e_{p-1,u_0}>>\; \cdots 
@>e_{3,u_0}>>\;
u_0(2)\; @>e_{2,u_0}>>\;  
u_0(1)\; @>e_{1,u_0}>>\; u_0(0) =u_0.  
\end{CD}
\end{equation*}
We define the shift ${\sigma}_{\widetilde{\frak L}}^k, k \in \Zp$
on $e_{[p]}x$ by 
\begin{equation*}
{\sigma}_{\widetilde{\frak L}}^k(e_{[p]}x)
=
\begin{cases}
e_{[p-k]}x & \text{ if }  p >k, \\
\sigma_{\frak L}^{k-p}(x) & \text{ if } p\le k.
\end{cases}
\end{equation*}
Define the groupoid
$G_{\widetilde{\frak L}}$ by setting
\begin{equation}
G_{\widetilde{\frak L}}
=\{
(e_{[p]}x, p-q, n, e_{[q]}z ) \mid
(x,n,z) \in G_{\frak L},  p,q \in \Zp \}.
\label{eq:GLtilde}
\end{equation}
The product and the inverse are defined by
\begin{equation*}
 (e_{[p]}x, p-q, n, e_{[q]}z )\cdot (e_{[p']}x', p'-q', n', e_{[q']}z' ) 
= 
(e_{[p]}x, p-q', n+n', e_{[q']}z' ) 
\end{equation*}
if $ q =p',\, z = x',$
and
\begin{equation*}
(e_{[p]}x, p-q, n, e_{[q]}z )^{-1} = (e_{[q]}z, q-p, -n, e_{[p]}x ).
\end{equation*}
The unit space 
$G_{\widetilde{\frak L}}^0$
is defined by
\begin{equation} \label{eq:Gtilde0}
G_{\widetilde{\frak L}}^0
=\{
(e_{[p]}x,0,0,e_{[p]}x)\in G_{\widetilde{\frak L}} \mid x \in X_{\frak L}, p \in \Zp\}.
\end{equation}
Define the groupoid homomorphism
$c_{{\widetilde{\frak L}}}:G_{\widetilde{\frak L}}\longrightarrow \Z$
by setting
\begin{equation} \label{eq:ctildel}
c_{{\widetilde{\frak L}}}(e_{[p]}x, p-q, n, e_{[q]}z ) = n-p+q,  \qquad 
(e_{[p]}x, p-q, n, e_{[q]}z )\in G_{\widetilde{\frak L}}.
\end{equation}
Let $G_{\Zp}$ be the \'etale groupoid $\Zp\times \Zp$ 
defined by the groupoid operations
\begin{gather*}
(m,n)\cdot(k,l) = (m,l) \text{ if } n=k, \text{ and}, \\
(m,n)^{-1} = (n,m) \text{ for } (m,n), (k,l) \in G_{\Zp}.
\end{gather*}
The set 
$G_{\Zp}$ is endowed with discrete topology. 
It is easy to see that the correspondence
\begin{equation*}
(e_{[p]}x, p-q, n, e_{[q]}z ) \in G_{\widetilde{\frak L}}
\longrightarrow
((x, n, z ), (p,q)) \in G_{\frak L} \times G_{\Zp} 
\end{equation*}
yields an isomorphism of groupoids.
Through this isomorphism, 
we may endow 
$G_{\widetilde{\frak L}}
$
with a topology of  
$G_{\widetilde{\frak L}}
$ 
induced from the product topology of the  \'etale groupoid
$G_{\frak L} \times G_{\Zp},$
so that  
$G_{\widetilde{\frak L}}$
is  isomorphic to
the product groupoid $G_{\frak L} \times G_{\Zp}$
as \'etale groupoids.
The unit space
$G_{\widetilde{\frak L}}^0$
is naturally identified with
$X_{\frak L}\times\Zp$.

Let $X_\Lambda$ be the one-sided subshift presented by ${\frak L}$.
For $\alpha = (\alpha_i)_{i\in \N} \in X_\Lambda$ and
$p \in \Zp$, 
we similarly define the head 
$e_{[p]}\alpha$
of $(\alpha_i)_{i\in \N}$
with length $p$ and define the space
\begin{equation*}
\widetilde{X}_\Lambda 
=\{e_{[p]}\alpha \mid \alpha = (\alpha_i)_{i\in \N} \in X_\Lambda, p \in \Zp\}
\end{equation*}  
that is identified with $X_\Lambda \times \Zp$
with its product topology.
Hence there exists a natural continuous surjection
\begin{equation}\label{eq:pitilde} 
\tilde{\pi}_{\frak L}: 
G_{\widetilde{\frak L}}^0 = X_{\frak L}\times\Zp
\longrightarrow 
\widetilde{X}_\Lambda = X_\Lambda\times \Zp
\end{equation}
that is defined by
$\tilde{\pi}_{\frak L}(x,m) = (\pi_{\frak L}(x),m)$ for $(x,m) \in X_{\frak L}\times\Zp.
$
We then see that there exists an inclusion relation 
\begin{equation*} 
C_0(\widetilde{X}_\Lambda) 
= C(X_\Lambda) \otimes C_0(\Zp) \hookrightarrow 
C(X_{\frak L}) \otimes C_0(\Zp) = C^*(G_{\widetilde{\frak L}}^0)
\end{equation*}
induced by 
$\tilde{\pi}_{\frak L}: 
G_{\widetilde{\frak L}}^0 \longrightarrow 
\widetilde{X}_\Lambda.
$

Let $f: G_{\frak L}\longrightarrow \Z$ 
be a continuous groupoid homomorphism.
Define a continuous groupoid homomorphism
$\tilde{f}: G_{\widetilde{\frak L}}\longrightarrow \Z$ by setting
\begin{equation*}
\tilde{f}(e_{[p]}x, p-q, n, e_{[q]}z ) = f(x, n, z )
\quad
\text{ for } 
(e_{[p]}x, p-q, n, e_{[q]}z ) \in G_{\widetilde{\frak L}}.
\end{equation*}
Recall that $\K$ denotes the
$C^*$-algebra of compact operators on the separable infinite dimensional Hilbert space
$\ell^2(\Zp)$ and
$\C (=C_0(\Zp))$ its maximal commutative $C^*$-subalgebra  
consisting of diagonal operators
on $\ell^2(\Zp)$.
Denote by 
$C^*(G_{\widetilde{\frak L}})$ 
the $C^*$-algebra of the \'etale groupoid $G_{\widetilde{\frak L}}.$

\begin{proposition}\label{prop:6.1}
There exists an isomorphism 
$\Theta:C^*(G_{\widetilde{\frak L}}) \longrightarrow
C^*(G_{\frak L} )\otimes \K$
of $C^*$-algebras such that 
$\Theta(C^*(G_{\widetilde{\frak L}}^0) )=C(X_{\frak L}) \otimes\C,$
$\Theta(C_0(\widetilde{X}_\Lambda) )=C(X_{\Lambda}) \otimes\C$
and
$\Theta\circ\rho^{\widetilde{\frak L},\tilde{f}}_t
=(\rho^{{\frak L},f}_t \otimes \id) \circ \Theta
$
for a continuous homomorphism
$f:G_{\widetilde{\frak L}}\longrightarrow \Z$ and $t \in \T$,
where
$C_0(\widetilde{X}_\Lambda) $ is regarded as a subalgebra of 
$C^*(G_{\widetilde{\frak L}}^0)$.
\end{proposition}
\begin{proof}
Let
$\{\theta_{p,q}\}_{p,q\in \Zp}$ be the matrix units 
of $\K$.
For $\mu =(\mu_1,\dots,\mu_k), \nu=(\nu_1,\dots,\nu_m)\in B_*(X_\Lambda)$
and
$v_i^l\in V_l$
with $k,m\le l$,
let 
$U(e_{[p]}\mu,v_i^l,e_{[q]}\nu)$
be the clopen set of $G_{\widetilde{\frak L}}$ defined by  
\begin{equation*}
U(e_{[p]}\mu,v_i^l, e_{[q]}\nu)
=\{  (e_{[p]}x, p-q, n, e_{[q]}z) \in G_{\widetilde{\frak L}}
\mid
(x,n,z) \in U(\mu, v_i^l, \nu) 
\}.
\end{equation*}
It is straightforward to see that
the correspondence
\begin{equation*}
\chi_{U(e_{[p]}\mu, v_i^l,e_{[q]}\nu)} \in C^*(G_{\widetilde{\frak L}})
\longrightarrow
S_\mu E_i^l S_\nu^* \otimes \theta_{p,q}
\in C^*(G_{\frak L} )\otimes \K
\end{equation*}
gives rise to an isomorphism 
$\Theta:C^*(G_{\widetilde{\frak L}}) \longrightarrow
C^*(G_{\frak L} )\otimes \K$
of $C^*$-algebras
satisfying the desired properties.
\end{proof}

Let us note that for a left-resolving $\lambda$-graph system
$\frak{L}$  satisfying condition (I), 
Lemma \ref{lem:masa} tells us that 
the subalgebra 
$(\mathcal{D}_\Lambda\otimes \C)^\prime \cap (\OFL\otimes\K)$
of elements of $\OFL\otimes\K$ 
commuting with all elements of 
$\mathcal{D}_\Lambda\otimes \C$
coincides with $\mathcal{D}_{\frak L}\otimes \C$, that is,
\begin{equation*}
(\mathcal{D}_\Lambda\otimes\C)^\prime 
\cap (\OFL\otimes\K) = \mathcal{D}_{\frak L}\otimes\C. 
\end{equation*}  
Hence if there 
exists an isomorphism
$\widetilde{\Phi}:{\mathcal{O}}_{{\frak L}_1}\otimes\K
\longrightarrow 
{\mathcal{O}}_{{\frak L}_2}\otimes\K
$ of $C^*$-algebras
such that 
$\widetilde{\Phi}({\mathcal{D}}_{{\Lambda}_1}\otimes\C)
={\mathcal{D}}_{{\Lambda}_2}\otimes\C,
$
then  
$\widetilde{\Phi}({\mathcal{D}}_{{\frak L}_1}\otimes\C)
={\mathcal{D}}_{{\frak L}_2}\otimes\C
$
holds.
By using Proposition \ref{prop:6.1},
the following proposition
can be proved in a similar way
as equivalence 
between
(ii) $\Longleftrightarrow $ (iii)
in Theorem \ref{thm:eventconj}.
Hence we omit its proof.
\begin{proposition}\label{prop:6.2}
Let ${\frak L}_1$ and 
 ${\frak L}_2$ be left-resolving $\lambda$-graph systems satisfying condition (I).
 Then the following are equivalent:
 \begin{enumerate}
\renewcommand{\theenumi}{\roman{enumi}}
\renewcommand{\labelenumi}{\textup{(\theenumi)}}
\item
There exist an isomorphism
$\tilde{\varphi}: G_{\widetilde{\frak L}_1}\longrightarrow 
G_{\widetilde{\frak L}_2}
$ of \'etale groupoids 
and a homeomorphism
$\tilde{h} :\widetilde{X}_{\Lambda_1} \longrightarrow \widetilde{X}_{\Lambda_2}
$
such that 
$\tilde{\pi}_{{\frak L}_2} 
\circ \tilde{\varphi}|_{G_{\widetilde{\frak L}_1}^0} = 
\tilde{h}\circ \tilde{\pi}_{{\frak L}_1}
$ 
and 
$c_{\widetilde{\frak L}_2}\circ\tilde{\varphi} =
c_{\widetilde{\frak L}_1}.$
\item
There exists an isomorphism
$\widetilde{\Phi}:{\mathcal{O}}_{{\frak L}_1}\otimes\K
\longrightarrow 
{\mathcal{O}}_{{\frak L}_2}\otimes\K
$ of $C^*$-algebras
such that 
\begin{equation*}
\widetilde{\Phi}({\mathcal{D}}_{{\Lambda}_1}\otimes\C)
={\mathcal{D}}_{{\Lambda}_2}\otimes\C, \qquad
\widetilde{\Phi}\circ (\rho^{{\frak L}_1}_t\otimes\id)
 =(\rho^{{\frak L}_2}_t\otimes\id)\circ \widetilde{\Phi}, \quad t \in \T.
\end{equation*} 
\end{enumerate}
\end{proposition}

\section{Two-sided conjugacy}
Let ${\frak L}$ 
be a left-resolving $\lambda$-graph system over $\Sigma$ satisfying condition (I).
Denote by $(X_{\frak L},\sigma_{\frak L})$
the associated topological dynamical system.
Define the shift space
$\bar{X}_{\frak L}$
of the topological dynamical system
 $(\bar{X}_{\frak L},\bar{\sigma}_{\frak L})$
by setting
\begin{equation}
\bar{X}_{\frak L} =\{
(\alpha_i,u_i) _{i \in \Z} \in \prod_{i\in \Z}(\Sigma\times\Omega_{\frak L})
\mid
(\alpha_{i+k},u_{i+k}) _{i \in \Z}
\in X_{\frak L}
 \text{ for all } k \in \Z\} \label{eq:barXfrakL}
\end{equation}
and
\begin{equation}
\bar{\sigma}_{\frak L}((\alpha_i,u_i) _{i \in \Z})=
 (\alpha_{i+1},u_{i+1})_{i \in \Z}. \label{eq:barsigmafrakL}
\end{equation}
We endow $\bar{X}_{\frak L}$ 
with the relative topology from the product topology
of $\prod_{i\in \Z}(\Sigma\times\Omega_{\frak L})$,
so that $\bar{X}_{\frak L}$ is a compact Hausdorff space.
Hence we have 
a topological dynamical system
$(\bar{X}_{\frak L},\bar{\sigma}_{\frak L})$
with a homeomorphism 
$\bar{\sigma}_{\frak L}$ on the compact Hausdorff space
$\bar{X}_{\frak L}$.
It is a two-sided extension of $(X_{\frak L},\sigma_{\frak L})$.
Let
$(X_\Lambda, \sigma_\Lambda)$ be the right one-sided subshift
presented by $\frak L$.
We then have the shift space
$\bar{X}_\Lambda$
and the homeomorphism
$\bar{\sigma}_\Lambda$ 
of the two-sided subshift
$(\bar{X}_\Lambda, \bar{\sigma}_\Lambda)$ 
in a similar way to
\eqref{eq:barXfrakL} and \eqref{eq:barsigmafrakL}.
The two-sided subshift $(\bar{X}_\Lambda, \bar{\sigma}_\Lambda)$ 
is written as $(\Lambda,\sigma)$ or $\Lambda$ for short.
For
$x =(\alpha_i,u_i)_{i\in \Z}\in \bar{X}_{\frak L}$,
$\alpha =(\alpha_i)_{i\in \Z}\in \Lambda$
and
$k, l \in \Z$ with $k<l$,
we set 
\begin{equation*}
x_{[k,l]} =(\alpha_i,u_i)_{i=k}^l,\quad
\alpha_{[k,l]} =(\alpha_i)_{i=k}^l,\quad
x_{[k,\infty)} =(\alpha_i,u_i)_{i=k}^\infty,\quad
\alpha_{[k,\infty)} =(\alpha_i)_{i=k}^\infty.
\end{equation*}
\begin{definition}\label{def:7.1}
Two topological dynamical systems
$(\bar{X}_{{\frak L}_1}, \bar{\sigma}_{{\frak L}_1})$
and
$(\bar{X}_{{\frak L}_2}, \bar{\sigma}_{{\frak L}_2})$
are said to be {\it right asymptotically topologically conjugate}\/ 
if there exists a topological conjugacy
$\psi:\bar{X}_{{\frak L}_1}\longrightarrow \bar{X}_{{\frak L}_2}$,
that is 
$\psi\circ\bar{\sigma}_{{\frak L}_1}=\bar{\sigma}_{{\frak L}_2}\circ\psi$,
such that 
\begin{enumerate}
\renewcommand{\theenumi}{\roman{enumi}}
\renewcommand{\labelenumi}{\textup{(\theenumi)}}
\item
for $m \in \Z$, there exists $M \in \Z$ such that 
$x_{[M,\infty)} = z_{[M,\infty)} $ implies 
$\psi(x)_{[m,\infty)} = \psi(z)_{[m,\infty)} $
for $x, z \in \bar{X}_{{\frak L}_1}$,
\item
for $n \in \Z$, there exists $N \in \Z$ such that 
$y_{[N,\infty)} = w_{[N,\infty)} $ implies 
$\psi^{-1}(y)_{[n,\infty)} = \psi^{-1}(w)_{[n,\infty)} $
for $y,w \in \bar{X}_{{\frak L}_2}$.
\end{enumerate}
In this case, we call
$\psi:\bar{X}_{{\frak L}_1}\longrightarrow \bar{X}_{{\frak L}_2}$
{\it a right asymptotic conjugacy.}\/
\end{definition}
For two $\lambda$-graph systems
${\frak L}_i, i=1,2$,
let us denote by $(\Lambda_i,\sigma_i)$
the two-sided subshift
$(\bar{X}_{\Lambda_i}, \bar{\sigma}_{\Lambda_i}).$ 
Let us denote by
$B_k(\Lambda_i)$ the set of admissible words $B_k(X_{\Lambda_i})$
of $\Lambda_i$ with length $k$. 

Let
$\bar{\pi}_i:\bar{X}_{{\frak L}_i}\longrightarrow \Lambda_i$
be the canonical factor map defined by
$
\bar{\pi}_i ((\alpha_i,u_i) _{i \in \Z}) = (\alpha_i)_{i\in\Z} \in\Lambda_i
$
for $i=1,2.$
The one-sided factor maps
$\pi_{{\frak L}_i}: X_{{\frak L}_i} \longrightarrow X_{\Lambda_i}, i=1,2$
are simply denoted by 
$\pi_i, i=1,2.$
\begin{definition}
The two subshifts 
$(\Lambda_1, \sigma_1)$ and 
$(\Lambda_2, \sigma_2)$
are said to be $({\frak L}_1,{\frak L}_2)$-{\it conjugate}\/ if
there exists a right asymptotic conjugacy
$\psi_{\frak L}:\bar{X}_{{\frak L}_1}\longrightarrow \bar{X}_{{\frak L}_2}$
and a topological conjugacy
$\psi_\Lambda:\Lambda_1\longrightarrow \Lambda_2$
such that 
the diagram
\begin{equation*}
\begin{CD}
\bar{X}_{{\frak L}_1} @>\psi_{\frak L} 
>> \bar{X}_{{\frak L}_2} \\
@V{\bar{\pi}_1 }VV  @VV{ \bar{\pi}_2}V \\
\Lambda_1 @> 
\psi_{\Lambda}>> 
\Lambda_2.
\end{CD}
\end{equation*}
commutes, that is $\bar{\pi}_2\circ\psi_{\frak L} = \psi_{\Lambda}\circ\bar{\pi}_1$.
\end{definition}
\begin{lemma}
Suppose that
$(\Lambda_1, \sigma_1)$ and 
$(\Lambda_2, \sigma_2)$
are $({\frak L}_1,{\frak L}_2)$-conjugate.
Keep the above notation. 
The topological conjugacies
$\psi_{\frak L}:\bar{X}_{{\frak L}_1}\longrightarrow \bar{X}_{{\frak L}_2}$
and 
$\psi_\Lambda:\Lambda_1\longrightarrow \Lambda_2$
satisfy the following properties:
\begin{enumerate}
\renewcommand{\theenumi}{\roman{enumi}}
\renewcommand{\labelenumi}{\textup{(\theenumi)}}
\item
There exists $M_1 \in \Z$ such that  

(a)
$x_{[m-M _1 ,\infty)} = z_{[m-M_1,\infty)} $ implies 
$\psi_{\frak L}(x)_{[m,\infty)} = \psi_{\frak L}(z)_{[m,\infty)} $
for $x, z \in \bar{X}_{{\frak L}_1}$ and $m \in \Z$,

(b)
$a_{[m-M _1 ,\infty)} = b_{[m-M_1,\infty)} $ implies 
$\psi_{\Lambda}(a)_{[m,\infty)} = \psi_{\Lambda}(b)_{[m,\infty)} $
for $a,b \in \Lambda_1$ and $m \in \Z$.
\item
There exists $N_1 \in \Z$ such that 

(a)
$y_{[n-N_1,\infty)} = w_{[n-N_1,\infty)} $ implies 
$\psi_{\frak L}^{-1}(y)_{[n,\infty)} = \psi_{\frak L}^{-1}(w)_{[n,\infty)} $
for $y,w \in \bar{X}_{{\frak L}_2}$ and  $n \in \Z$, 

(b)
$c_{[n-N _1 ,\infty)} = d_{[n-N_1,\infty)} $ implies 
$\psi_{\Lambda}^{-1}(c)_{[n,\infty)} = \psi_{\Lambda}^{-1}(d)_{[n,\infty)} $
for $c,d \in \Lambda_2$ and $n \in \Z$.
\end{enumerate}
\end{lemma}
\begin{proof}
The assertions concerning 
$\psi_{\frak L},\psi_{\frak L}^{-1}$ follow from the conditions (i) and (ii) of 
Definition \ref{def:7.1} with the condition 
$\psi\circ\bar{\sigma}_{{\frak L}_1}=\bar{\sigma}_{{\frak L}_2}\circ\psi$.
The assertions concerning 
$\psi_{\Lambda}, \psi_{\Lambda}^{-1}$
follow from the fact that 
$\psi_{\Lambda}, \psi_{\Lambda}^{-1}$
are sliding block codes (cf. \cite{LM}).
\end{proof}
Recall that 
$\widetilde{\frak L}$ denote the stabilization of ${\frak L}$.
The following theorem combined with Proposition \ref{prop:6.2}
implies Theorem \ref{thm:1.4}.
\begin{theorem}\label{thm:twosidedconjugacy1}
Let ${\frak L}_1$ and 
 ${\frak L}_2$ be left-resolving $\lambda$-graph systems satisfying condition (I).
Let
$(\Lambda_1,\sigma_1)$ and
$(\Lambda_2,\sigma_2)$
be their two-sided subshifts presented by 
${\frak L}_1$ and ${\frak L}_2$, respectively.
Then the following assertions are equivalent:
 \begin{enumerate}
\renewcommand{\theenumi}{\roman{enumi}}
\renewcommand{\labelenumi}{\textup{(\theenumi)}}
\item
The two subshifts 
$(\Lambda_1, \sigma_1)$ and 
$(\Lambda_2, \sigma_2)$
are $({\frak L}_1,{\frak L}_2)$-conjugate.
\item
There exist an isomorphism
$\tilde{\varphi}: G_{\widetilde{\frak L}_1}\longrightarrow 
G_{\widetilde{\frak L}_2}
$ of \'etale groupoids 
and a homeomorphism
$\tilde{h} :\widetilde{X}_{\Lambda_1} \longrightarrow \widetilde{X}_{\Lambda_2}
$
such that 
$\tilde{\pi}_{{\frak L}_2} 
\circ \tilde{\varphi}|_{G_{\widetilde{\frak L}_1}^0} = 
\tilde{h}\circ \tilde{\pi}_{{\frak L}_1}
$ 
and 
$c_{\widetilde{\frak L}_2}\circ\tilde{\varphi} =
c_{\widetilde{\frak L}_1}.$
\end{enumerate}
\end{theorem}
We will first prove (i) $\Longrightarrow$ (ii).
Suppose that the subshifts 
$(\Lambda_1, \sigma_1)$ and 
$(\Lambda_2, \sigma_2)$
are $({\frak L}_1,{\frak L}_2)$-conjugate.
There exists 
a right asymptotic conjugacy
$\psi_{\frak L}:\bar{X}_{{\frak L}_1}\longrightarrow\bar{X}_{{\frak L}_2}$
and a topological conjugacy
$\psi_\Lambda:\Lambda_1\longrightarrow \Lambda_2$
such that 
 $\pi_2\circ\psi_{\frak L} = \psi_{\Lambda}\circ\pi_1$.
One may find $M \in \Z$ such that 
if $(x_n)_{n \in \N} = (x'_n)_{n\in \N}$,  
then 
$\psi_{\frak L}(x)_{[M,\infty)} = \psi_{\frak L}(x')_{[M,\infty)}$
for
$x =(x_n)_{n\in \Z}, x' =(x'_n)_{n\in \Z} \in \bar{X}_{{\frak L}_1}.$
By taking 
$\bar{\sigma}_{{\frak L}_2}^M\circ\psi_{\frak L}$ 
instead of $\psi_{\frak L}$,
we may assume that $M=0$.
We then find $l \in \N$ such that 
if for $(y_n)_{n\in \N} =(y'_n)_{n\in \N}\in X_{{\frak L}_2}$,
then
$\psi_{\frak L}^{-1}(y)_{[l,\infty)} =\psi_{\frak L}^{-1}(y')_{[l,\infty)}$
for $y =(y_n)_{n\in \Z}, y' =(y'_n)_{n\in \Z}\in \bar{X}_{{\frak L}_2}.$
Hence there exists a continuous surjection
$\psi_0:X_{{\frak L}_1}\longrightarrow X_{{\frak L}_2}$
such that 
$
\psi_0((x_n)_{n\in\N}) = (\psi_{\frak L}(x)_n)_{n\in\N}
$
for $x =(x_n)_{n\in\Z}\in \bar{X}_{{\frak L}_1}$
such that 
$\psi_0\circ\sigma_{{\frak L}_1} = 
\sigma_{{\frak L}_2}\circ\psi_0$.
It has a property such that if
$\psi_0(x) = \psi_0(x')$ for $x,x' \in X_{{\frak L}_1}$,
then $x_{[l,\infty)}= x'_{[l,\infty)}.$
Since 
$\psi_\Lambda:\Lambda_1\longrightarrow \Lambda_2$ is a topological conjugacy,
it is a sliding block code (cf.  \cite{LM}) 
so that we can find $L\ge l$ such that 
\begin{equation*}
a_{[1,L]} =b_{[1,L]} \text{ implies } 
\psi_\Lambda(a)_{[1,l]} =\psi_\Lambda(b)_{[1,l]}
\text{ for } a, b \in \Lambda_1.
\end{equation*}
Let $v_1^L,\dots,v_{m(L)}^L$
be the set of vertices $V_L$.
Let
$B_L(\Lambda_1, v_i^L), i=1,\dots,m(L)$ 
be the set of admissible words of $\Lambda_1$
of length $L$ defined by
\begin{align*}
B_L(\Lambda_1, v_i^L)
=\{(\alpha_i)_{i=1}^L\in B_L(\Lambda_1) 
\mid
\text{ there exists }
(\alpha_n,u_n)_{n\in \N}\in X_{{\frak L}_1},\,  u_L^L =v_i^L \}
\end{align*}
where $u_L =(u_L^n)_{n\in\Zp} \in \Omega_{{\frak L}_1}.$
The set  
$B_L(\Lambda_1, v_i^L)$
consists of admissible words of the subshift $\Lambda_1$
of length $L$ terminating at the vertex $v_i^L$.

We fix a vertex $v_i^L \in V_L$ for a while.
Consider a relation $\sim$ on 
$B_L(\Lambda_1)$ in the following way.
Two words
$\mu, \mu'\in B_L(\Lambda_1, v_i^L)$ 
are written $\mu\sim\mu'$ if 
there exist 
$
x =(\alpha_n,u_n)_{n\in \N}, 
x' =(\alpha'_n,u'_n)_{n\in \N} \in X_{\frak L}   
$
such that 
$$
\mu=(\alpha_n)_{n=1}^L, \,\,
 \mu'=(\alpha'_n)_{n=1}^L,\,\,
u_L^L = {u'}_L^L =v_i^L \,
\text{ 
and } \,
\psi_0(x) = \psi_0(x'),
$$

\noindent
where 
$u_L = (u_L^n)_{n \in \Zp}, u'_L = ({u'}_L^n)_{n \in \Zp}\in \Omega_{{\frak L}_1}.$

In order to prove implication (i) $\Longrightarrow$ (ii),
we first show  Lemma \ref{lem:7.4}, Lemma \ref{lem:7.5} and Lemma \ref{lem:7.6}.  
\begin{lemma}\label{lem:7.4}
$\sim$ is an equivalence relation on $B_L(\Lambda_1,v_i^L)$.
\end{lemma}
\begin{proof}
Suppose that $\mu \sim \mu'$ and $\mu'\sim\mu''$
in $B_L(\Lambda_1,v_i^L).$ 
Take 
$
x =(\alpha_n,u_n)_{n\in \N},\, 
x' =(\alpha'_n,u'_n)_{n\in \N} \in X_{{\frak L}_1}   
$
such that 
$\mu=(\alpha_n)_{n=1}^L,
 \mu'=(\alpha'_n)_{n=1}^L,
$ 
$u_L^L = {u'}_L^L =v_i^L$
and
$\psi_0(x) = \psi_0(x').$
Similarly by the condition 
$\mu'\sim\mu''$,
take
$
z =(\beta_n,w_n)_{n\in \N}, \,
z' =(\beta'_n,w'_n)_{n\in \N}\in X_{{\frak L}_1}   
$
such that 
$\mu'=(\beta_n)_{n=1}^L,\,
 \mu''=(\beta'_n)_{n=1}^L,\,
$ 
$w_L^L = {w'}_L^L =v_i^L$
and
$\psi_0(z) = \psi_0(z').$
Since ${\frak L}_1$ is left-resolving,
the condition 
$\mu' =(\alpha'_n)_{n=1}^L =(\beta_n)_{n=1}^L$
and $L\ge l$ 
implies that 
 there exist unique 
$\bar{x}, \bar{z} \in X_{{\frak L}_1}$ such that 
\begin{equation*}
\pi_1(\bar{x})_{[1,L]} =\mu,\quad
\bar{x}_{[L+1,\infty)} = x'_{[L+1,\infty)}
\quad\text{ and }\quad
\pi_1(\bar{z})_{[1,L]} =\mu'',\quad
\bar{z}_{[L+1,\infty)} = x'_{[L+1,\infty)}.
\end{equation*}
By the conditions 
$\psi_0(x) = \psi_0(x')$
and
$\psi_0(z) = \psi_0(z'),$
we have
$x_{[l,\infty)} =x'_{[l,\infty)}$
and
$z_{[l,\infty)} =z'_{[l,\infty)},$
 respectively.
Hence we have
$u_L = u'_L$ in $\Omega_{{\frak L}_1}$
and
$\mu_{[l,L]} =\mu'_{[l,L]}$,
and similarly
$w_L = w'_L$ in $\Omega_{{\frak L}_1}$
and
$\mu'_{[l,L]} =\mu''_{[l,L]}$,
so that 
\begin{equation}
\mu_{[l,L]} = \mu''_{[l,L]}. \label{eq:6.4.5}
\end{equation}
As $x_{[l,\infty)} =x'_{[l,\infty)}$
and $l\le L$, we see that
$x_{[L+1,\infty)} =x'_{[L+1,\infty)}$, 
so that we have
\begin{equation*}
\bar{z}_{[l,\infty)} 
=\bar{z}_{[l,L]}x'_{[L+1,\infty)} 
=\bar{z}_{[l,L]}x_{[L+1,\infty)}.
\end{equation*}
Since
$\pi_1(\bar{z}_{[l,L]}) =\mu''_{[l,L]},
 \pi_1(x_{[l,L]}) =\mu_{[l,L]}
$
with
\eqref{eq:6.4.5}
and
${{\frak L}_1}$ is left-resolving,
we see that
$\bar{z}_{[l,L]}x_{[L+1,\infty)} =x_{[l,L]}x_{[L+1,\infty)}
= x_{[l,\infty)}$.
Hence we have
\begin{equation*}
\bar{z}_{[l,\infty)} =x_{[l,\infty)}. 
\end{equation*}
As ${{\frak L}_1}$ is left-resolving
and
$\pi_1(\bar{x})_{[1,L]} = \mu=\pi_1(x)_{[1,L]},$
we have
\begin{equation*}
\bar{x} 
=\bar{x}_{[1,L]}x'_{[L+1,\infty)} 
=\bar{x}_{[1,L]}x_{[L+1,\infty)} 
=x_{[1,L]} x_{[L+1,\infty)}
=x. 
\end{equation*}
We also have
\begin{gather*}
\psi_0(\bar{x})_{[l,\infty)} 
=\sigma_{{\frak L}_2}^{l-1}(\psi_0(\bar{x}))
=\psi_0(\sigma_{{\frak L}_1}^{l-1}(\bar{x}))
=\psi_0(\bar{x}_{[l,\infty)}), \\
\psi_0(\bar{z})_{[l,\infty)} 
=\sigma_{{\frak L}_2}^{l-1}(\psi_0(\bar{z}))
=\psi_0(\sigma_{{\frak L}_1}^{l-1}(\bar{z}))
=\psi_0(\bar{z}_{[l,\infty)}).
\end{gather*}
Since
$\bar{x} =x$ and 
$x_{[l,\infty)}=\bar{z}_{[l,\infty)},
$ 
we have
\begin{equation} 
\psi_0(\bar{x})_{[l,\infty)}
=
\psi_0(\bar{z})_{[l,\infty)}. \label{eq:6.4.8}
\end{equation}
And also
\begin{equation*}
\pi_1(\bar{x})_{[1,L]}  =\mu = \pi_1(x)_{[1,L]}, \qquad
\pi_1(\bar{z})_{[1,L]}  =\mu'' = \pi_1(z')_{[1,L]}. 
\end{equation*}
Hence we have
\begin{equation}
\psi_\Lambda(\pi_1(\bar{x}))_{[1,l]} =\psi_\Lambda(\pi_1({x}))_{[1,l]}, \qquad
\psi_\Lambda(\pi_1(\bar{z}))_{[1,l]} =\psi_\Lambda(\pi_1(z'))_{[1,l]}. \label{eq:6.4.11}
\end{equation}
It then follows that
\begin{align*}
\pi_2(\psi_0(\bar{x}))_{[1,l]}
=& \psi_\Lambda(\pi_1(\bar{x}))_{[1,l]}
= \psi_\Lambda(\pi_1({x}))_{[1,l]} \quad (\text{by }\eqref{eq:6.4.11}) \\
=& \pi_2(\psi_0({x}))_{[1,l]} 
= \pi_2(\psi_0(x'))_{[1,l]} \quad (\text{by }\psi_0(x) = \psi_0(x'))\\
=& \psi_\Lambda(\pi_1({x'}))_{[1,l]}  
= \psi_\Lambda(\pi_1({z'}))_{[1,l]} \quad 
(\text{by } \pi_1(x')_{[1,L]} =\mu' =\pi_1(z')_{[1,L]}) \\
=& \psi_\Lambda(\pi_1(\bar{z}))_{[1,l]}  
= \pi_2(\psi_0(\bar{z}))_{[1,l]}
\end{align*}
so that
\begin{equation}
\pi_2(\psi_0(\bar{x}))_{[1,l]}
=
\pi_2(\psi_0(\bar{z}))_{[1,l]}. \label{eq:6.4.13}
\end{equation}
By \eqref{eq:6.4.8}
and \eqref{eq:6.4.13}, 
we reach the equality 
$\psi_0(\bar{x})
=\psi_0(\bar{z})
$ 
because ${\frak L}_1$ is left-resolving.
Since
$\bar{x}_{[1,L]} =\mu,
\bar{z}_{[1,L]} =\mu'',$
we conclude that 
$\mu\sim\mu''$
so that  
$\sim$ is an equivalence relation on $B_L({\Lambda_1}, v_i^l)$.
\end{proof}

\begin{lemma}\label{lem:7.5}
 For $x,x' \in X_{{\frak L}_1}$, we have 
 $\psi_0(x) = \psi_0(x')$ 
if and only if
$x_{[l,\infty)} =x'_{[l,\infty)}$
and
$\pi_1(x)_{[1,L]}\sim \pi_1(x')_{[1,L]}$ in $B_L(\Lambda_1,v_i^L)$
for some  $v_i^L\in V_L.$ 
\end{lemma}
\begin{proof}
Let $x =(\alpha_n,u_n)_{n\in\N}, \, x' =(\alpha'_n,u'_n)_{n\in\N}\in X_{{\frak L}_1}.$
We will first show  the only if part.
Suppose that
$\psi_0(x) = \psi_0(x')$.
Hence we have
$x_{[l,\infty)} =x'_{[l,\infty)}.$
As $L\ge l$, we have
$u_L = u'_L$ and
particularly
we may put
$v_i^L=u_L^L = u_L^{\prime L} \in V_L$,
where
$u_L = (u_L^n)_{n \in \Zp},u'_L = (u_L^{\prime n})_{n \in \Zp}.$
Hence
$\pi_1(x)_{[1,L]},
 \pi_1(x')_{[1,L]}\in B_L(\Lambda_1,v_i^L).$
Since
$\psi_0(x) = \psi_0(x')$,
we see that 
$\pi_1(x)_{[1,L]}\sim \pi_1(x')_{[1,L]}$ in $B_L(\Lambda_1,v_i^l).$

We will next show the if part as follows.
Suppose that 
$\pi_1(x)_{[1,L]}\sim \pi_1(x')_{[1,L]}$ 
and 
$x_{[l,\infty)} =x'_{[l,\infty)}.$
One may find $z, z' \in X_{{\frak L}_1}$ 
such that 
$\pi_1(z)_{[1,L]} =\pi_1(x)_{[1,L]}, \, \pi_1(z')_{[1,L]} =\pi_1(x')_{[1,L]}$
and
$\psi_0(z) = \psi_0(z')$.
Hence we have
$z_{[l,\infty)} =z'_{[l,\infty)},$
and
\begin{align*}
\psi_0(x)_{[l,\infty)} 
=& \sigma_{{\frak L}_2}^{l-1}(\psi_0(x))
=\psi_0(\sigma_{{\frak L}_1}^{l-1}(x))
=\psi_0(\sigma_{{\frak L}_1}^{l-1}(x'))
=\sigma_{{\frak L}_2}^{l-1}(\psi_0(x'))
=\psi_0(x')_{[l,\infty)}, \\
\pi_2(\psi_0(x))_{[1,l]} 
=&\psi_\Lambda(\pi_1(x))_{[1,l]} 
=\psi_\Lambda(\pi_1(z))_{[1,l]} 
=\pi_2(\psi_0(z))_{[1,l]} \\
=&\pi_2(\psi_0(z'))_{[1,l]}
=\psi_\Lambda(\pi_1(z'))_{[1,l]}
=\psi_\Lambda(\pi_1(x'))_{[1,l]}
=\pi_2(\psi_0(x'))_{[1,l]}. 
\end{align*}
Since ${{\frak L}_2}$ is left-resolving, 
we have
$\psi_0(x) = \psi_0(x')$.
\end{proof}


Consider the set of equivalence classes $B_L(\Lambda_1, v_i^L)/\sim.$
For $\mu =(\mu_1,\dots,\mu_L)\in B_L(\Lambda_1,v_i^L)$,
denote by
$[\mu]_{v_i^L} $ 
the equivalence class of $\mu$
$$
[\mu]_{v_i^L} = \{ \nu \in B_L(\Lambda_1,v_i^L)\mid
\nu\sim\mu\}.
$$
Under fixing 
 the class $[\mu]_{v_i^L}$, 
 we decompose the set 
$\Z_+$ of nonnegative integers into disjoint subsets 
$\Z_+(\nu;v_i^L), \nu \in [\mu]_{v_i^L}$
such that
$\Z_+ =\bigcup_{\nu \in [\mu]_{v_i^L}}\Z_+(\nu;v_i^L)$
and
there exists a bijection
\begin{equation*}
g_{(\nu,v_i^L)}:\Z_+\longrightarrow \Z_+(\nu,v_i^L)
\quad
\text{ for each }
\quad \nu \in [\mu]_{v_i^l}.
\end{equation*}
For $x = (\alpha_n,u_n)_{n\in \N}\in X_{{\frak L}_1}$,
we write
\begin{equation*}
\lambda_n(x) = \alpha_n\in \Sigma,
\qquad
v_n(x) = u_n = (u_n^l)_{l\in \Zp} \in \Omega_{\frak L}
\quad
\text{ and }
\quad
v_n^l(x) = u_n^l \in V_l.
\end{equation*}
Hence
$v_L^L(x)=u_L^L \in V_L$, and
we know that
$\pi_1(x)_{[1,L]} \in B_L(\Lambda_1,v_L^L(x)).$
Let us denote by
$g_{(x,L)}$
the map
$
g_{(\pi_1(x)_{[1,L]}, v_L^L(x))}:\Zp \longrightarrow \Z_+(\pi_1(x)_{[1,L]},v_L^L(x)).
$
\begin{lemma}\label{lem:7.6}
The  map:
\begin{equation*}
\xi: (x,n) \in X_{{\frak L}_1} \times\Zp \longrightarrow 
(\psi_0(x), g_{(x,L)}(n)) \in X_{{\frak L}_2} \times\Zp
\end{equation*}
is a homeomorphism between 
$ X_{{\frak L}_1} \times\Zp$ and $ X_{{\frak L}_2} \times\Zp$.
\end{lemma}
\begin{proof}
1. Injectivity of $\xi$:

Suppose that 
$\xi(x,n) = \xi(x',n')$,
so that 
$\psi_0(x) =\psi_0(x')$ and
$g_{(x,L)}(n) =g_{(x',L)}(n').$
As 
$\psi_0(x) =\psi_0(x')$, we have
$v_L^L(x) = v_L^L(x')$ and put the vertex as 
$v_i^L$. 
By the preceding lemma, we know that 
$\pi_1(x)_{[1,L]} \sim \pi_1(x')_{[1,L]}$
in $B_L(\Lambda_1,v_i^L)$.
If $\pi_1(x)_{[1,L]} \ne \pi_1(x')_{[1,L]},$
we see that
$g_{(x,L)}(\Z_+) \cap g_{(x',L)}(\Z_+') =\emptyset,$
so that 
$g_{(x,L)}(n) \ne g_{(x',L)}(n').$
Hence we get
$\pi_1(x)_{[1,L]} = \pi_1(x')_{[1,L]}.$
As 
$g_{(x,L)}:\Z_+\longrightarrow \Z_+(\pi_1(x)_{[1,L]},v_i^L)$
is bijective,
we obtain that $n=n'$
because $g_{(x,L)}(n) =g_{(x',L)}(n').$
And also by the preceding lemma 
the condition
$\psi_0(x) =\psi_0(x')$
implies that
$x_{[l,\infty)} =x'_{[l,\infty)}$.
Since $l\le L$, we have
$\pi_1(x)_{[1,l]} = \pi_1(x')_{[1,l]},$
so that we have $x =x'$.
This shows that 
$\xi:X_{{\frak L}_1} \times\Z_+\longrightarrow 
X_{{\frak L}_2} \times\Z_+
$
is injective.
 
\medskip

2. Surjectivity of $\xi$:

Concerning surjectivity of $\xi$, take an arbitrary element
$(y,m) \in X_{{\frak L}_2}\times\Zp$.
Since $\psi_0:X_{{\frak L}_1}\longrightarrow X_{{\frak L}_2}$ 
is surjective,
one may find $x\in X_{{\frak L}_1}$ such that 
$\psi_0(x) =y$.
For $\pi_1(x)_{[1,L]}\in B_L(\Lambda_1,v_L^L(x))$,
the set 
$\Zp$ is decomposed into disjoint union
$$
\Zp = \bigcup_{\nu \in [{\pi_1(x)}_{[1,L]}]_{v_L^L(x)}}\Zp(\nu,v_L^L(x)). 
$$
One may find $\nu\in [{\pi_1(x)}_{[1,L]}]_{v_L^L(x)}$ 
such that 
$m \in  \Zp(\nu,v_L^L(x))$.
There exist
$z,z' \in X_{{\frak L}_1}$ such that 
\begin{gather}
\psi_0(z) = \psi_0(z'), \quad
\pi_1(z)_{[1,L]} =\nu, \quad
\pi_1(z')_{[1,L]} =\pi_1(x)_{[1,L]}, \label{eq:6.6.2}\\
v_L^L(z) = v_L^L(z') = v_L^L(x). \label{eq:6.6.3}
\end{gather}
As $\pi_1(z)_{[1,L]} \in B_L(\Lambda_1,v_L^L(x))$,
we may find $\bar{x} \in X_{{\frak L}_1}$
such that 
\begin{gather}
\pi_1(\bar{x})_{[1,L]} =\pi_1(z)_{[1,L]} (=\nu), \label{eq:6.6.4}\\
\bar{x}_{[L+1,\infty)} = x_{[L+1,\infty)}. \label{eq:6.6.5}
\end{gather}
It then follows that
\begin{equation}
\psi_0(\bar{x})_{[l,\infty)}
=\sigma_{{\frak L}_2}^{l-1}(\psi_0(\bar{x})) 
=\psi_0(\sigma_{{\frak L}_1}^{l-1}(\bar{x})) 
=\psi_0(x_{[l,\infty)}).\label{eq:6.6.6}
\end{equation}
Now
the condition $\psi_0(z) =\psi_0(z')$
implies 
$z_{[l,\infty)}= z'_{[l,\infty)}$
so that  
$\pi_1(z)_{[l,\infty)}= \pi_1(z')_{[l,\infty)}.$
Hence by \eqref{eq:6.6.4} and \eqref{eq:6.6.5}, we have 
\begin{equation}
\pi_1(\bar{x})_{[l,L]} 
=\pi_1(z)_{[l,L]}
=\pi_1(z')_{[l,L]}
=\pi_1(x)_{[l,L]}. \label{eq:6.6.8}
\end{equation}
As ${\frak L}_1$ is left-resolving,
the equality \eqref{eq:6.6.8}
implies that
$\bar{x}_{[l,L]} {x}_{[L+1,\infty)} 
={x}_{[l,L]} {x}_{[L+1,\infty)} 
$
so that 
\begin{equation}
\bar{x}_{[l,\infty)} 
=\bar{x}_{[l,L]} 
\bar{x}_{[L+1,\infty)} 
=\bar{x}_{[l,L]} 
 {x}_{[L+1,\infty)} 
={x}_{[l,\infty)}. \label{eq:6.6.9} 
\end{equation}
By \eqref{eq:6.6.6} and \eqref{eq:6.6.9},
we have
\begin{equation}
\psi_0(\bar{x})_{[l,\infty)}
=\psi_0({x}_{[l,\infty)})
=\psi_0(\sigma_{{\frak L}_1}^{l-1}({x})) 
=\sigma_{{\frak L}_2}^{l-1}(\psi_0({x})) 
=\psi_0({x})_{[l,\infty)}.\label{eq:6.6.10}
\end{equation}
By \eqref{eq:6.6.2} and \eqref{eq:6.6.4}
we have
$\psi_\Lambda(\pi_1(z'))_{[1,l]} = \psi_\Lambda(\pi_1(x))_{[1,l]}$
and
$\psi_\Lambda(\pi_1(\bar{x}))_{[1,l]} = \psi_\Lambda(\pi_1(z))_{[1,l]}$,
respectively,
so that 
\begin{align*}
\pi_2(\psi_0(\bar{x}))_{[1,l]}
=& \psi_\Lambda(\pi_1(\bar{x}))_{[1,l]}
= \psi_\Lambda(\pi_1({z}))_{[1,l]}  \\
=& \pi_2(\psi_0({z}))_{[1,l]} 
= \pi_2(\psi_0(z'))_{[1,l]} \quad (\text{by }\psi_{\frak L}(z) = \psi_{\frak L}(z'))\\
=& \psi_\Lambda(\pi_1({z'}))_{[1,l]}  
= \psi_\Lambda(\pi_1({x}))_{[1,l]}  \\
=& \pi_2(\psi_0({x}))_{[1,l]}
\end{align*}
so that
\begin{equation}
\pi_2(\psi_0(\bar{x}))_{[1,l]}
=
\pi_2(\psi_0(x))_{[1,l]}. \label{eq:6.6.11}
\end{equation}
Since ${\frak L}_2$ is left-resolving,
the equalities 
 \eqref{eq:6.6.10} and \eqref{eq:6.6.11} ensure that 
\begin{equation*}
\psi_0(\bar{x})= \psi_0(x). 
\end{equation*}
 Therefore we get
 $\bar{x} \in X_{{\frak L}_1}$ such that 
\begin{equation*}
\psi_0(\bar{x}) = y,\qquad 
\pi_1(\bar{x})_{[1,L]} =\nu, \qquad
\pi_1(\bar{x})_{[1,L]} \in B_L(\Lambda_1,v_L^L(x)),
\end{equation*}
where
$\nu$ satisfies $m \in \Zp(\nu,v_L^L(x))$.
Hence we have
$m \in \Zp(\pi_1(\bar{x})_{[1,L]},v_L^L(\bar{x})).$
Since
$g_{(\bar{x}, L)}:\Zp\longrightarrow \Zp(\pi_1(\bar{x})_{[1,L]},v_L^L(\bar{x}))$
is bijective,
one may find $n \in \Zp$ such that 
$g_{(\bar{x},L)}(n) = m$.
We thus have
\begin{equation*}
(y,m) = (\psi_0(\bar{x}), g_{(\bar{x},L)}(n) ) = \xi(\bar{x}, n)
\end{equation*}
and hence 
$\xi$ is surjective.

\medskip

3. Continuity of $\xi$:

Suppose that $(x,n)$ and $(x',n')$ are in a small open neighborhood in $X_{{\frak L}_1}\times\Zp$.
We see that 
$\pi_1(x)_{[1,L]}=\pi_1(x')_{[1,L]}$ and
$v_L^L(x) =v_L^L(x')$ and $n=n'$.
Hence we have
$g_{(x,L)} =g_{(x',L)}$ and particularly
$g_{(x,L)}(n) =g_{(x',L)}(n').$
Since $x$ is close to $x'$, 
by the continuity of $\psi_0$,
$\psi_0(x)$ is close to $\psi_0(x').$
Hence we know that 
$\xi:X_{{\frak L}_1} \times\Z_+\longrightarrow 
X_{{\frak L}_2} \times\Z_+
$
is continuous.


Since $\psi_{\frak L}:\bar{X}_{{\frak L}_1}\longrightarrow \bar{X}_{{\frak L}_2}$
is a homeomorphism,
the continuous map
$\psi_0:{X}_{{\frak L}_1}\longrightarrow {X}_{{\frak L}_2}$
is a local homeomorphism by its construction.
For any $(x,n) \in X_{{\frak L}_1}\times\Zp$,
one may take an open neighborhood
$U(\pi_1(x)_{[1,L]}, v_L^L(x))$
of $x$ so that
$g_{(\pi_1(z)_{[1,L]}, v_L^L(z))} =g_{(\pi_1(x)_{[1,L]}, v_L^L(x))}$
for $z \in U(\pi_1(x)_{[1,L]}, v_L^L(x)).$
Hence we know that 
$$
\xi(U(\pi_1(x)_{[1,L]}, v_L^L(x)), n)
=(\psi_0(U(\pi_1(x)_{[1,L]}, v_L^L(x))), n')
$$
where
$n' =g_{(\pi_1(x)_{[1,L]}, v_L^L(x))}(n).$
Since
$\psi_0$ is a local homeomorphism,
$\psi_0(U(\pi_1(x)_{[1,L]}, v_L^L(x)))$ is open in $X_{{\frak L}_2}$
so that 
$
\xi(U(\pi_1(x)_{[1,L]}, v_L^L(x)), n)
$
is open in $X_{{\frak L}_2} \times \Zp$.
This shows that 
$\xi$ maps an open set of
$X_{{\frak L}_1} \times \Zp$
to an open set of
$X_{{\frak L}_2} \times \Zp$,
so that 
it is a homeomorphism.
\end{proof}
We will pass to the proof of  implication (i) $\Longrightarrow$ (ii)
of Theorem \ref{thm:twosidedconjugacy1}.
Basic strategy of the proof below as well as the proof of implication 
(ii) $\Longrightarrow$ (i) comes from the proof of Theorem 5.1 in 
Carlsen--Rout's paper \cite{CR}.
  
Recall that $\widetilde{\frak L}_i$ denote
the stabilizations of ${\frak L}_i$.
We write the groupoids
$$
G_{\widetilde{\frak L}_i} = G_{{\frak L}_i} \times G_{\Zp}
=\{((x,p), n,(z,q)) \mid (x, n+q-p, z) \in G_{{\frak L}_i}, (p,q) \in G_{\Zp} \},
\quad i=1,2
$$
where $G_{\Zp} = \Zp\times\Zp.$
Although the set of the  right hand side above looks different from
\eqref{eq:GLtilde}, it is actually isomorphic to the groupoid
$G_{\widetilde{\frak L}_i}$ defined by \eqref{eq:GLtilde},
because the correspondence 
\begin{equation*}
(e_{[p]}x, p-q, m, e_{[q]}z ) 
\longrightarrow 
\{((x,p), m+p-q, (z,q)) \mid (x,m,z) \in G_{{\frak L}_i}, (p,q) \in G_{\Zp} \}
\end{equation*}
from $G_{\widetilde{\frak L}_i}$
to the  right hand side above
gives rise to an isomorphism of \'etale groupoids.

Consider the following map
\begin{equation*}
\tilde{\varphi}:
 ((x,p), n,(z,q))\in G_{\widetilde{\frak L}_1}
\longrightarrow 
 (\xi(x,p),
n+q +g_{(x,L)}(p) -p-g_{(z,L)}(q),
\xi(z,q)) \in G_{\widetilde{\frak L}_2},
\end{equation*}
where $\xi: X_{{\frak L}_1} \times\Zp \longrightarrow X_{{\frak L}_2} \times\Zp$
is the homeomorphism defined in Lemma \ref{lem:7.6}.
We will see that the map above is a well-defined isomorphism of \'etale groupoids.
We denote $(x,p)$
by
$e_{[p]}x. $  
Let 
$((x,p), n,(z,q))\in G_{\widetilde{\frak L}_1}$,
which means that there exist
$k,l\in \Zp$ such that 
$k-l =n,$ and
$ 
\sigma_{\widetilde{\frak L}_1}^k(e_{[p]}x) 
=\sigma_{\widetilde{\frak L}_1}^l(e_{[q]}z).
$
We may assume that $k\ge p, \, l\geq q$.
Since 
$ 
\sigma_{\widetilde{\frak L}_1}^k(e_{[p]}x) 
=\sigma_{{\frak L}_1}^{k-p}(x)
$
and
$ 
\sigma_{\widetilde{\frak L}_1}^l(e_{[q]}z) 
=\sigma_{{\frak L}_1}^{l-q}(z)$,
the condition
$((x,p), n,(z,q))\in G_{\widetilde{\frak L}_1}$
is equivalent to the condition that 
there exist
$m(=k-p),m'(=l-q) \in \Zp$ 
such that 
$m+p -m'-q =n$ 
and
$\sigma_{{\frak L}_1}^{m}(x)
=\sigma_{{\frak L}_1}^{m'}(z)$.
As 
$\xi(x,p) = (\psi_0(x), g_{(x,L)}(p))$ 
and
$\xi(z,q) =(\psi_0(z), g_{(z,L)}(q)),$
we have
\begin{align}
\sigma_{\widetilde{\frak L}_2}^{g_{(x,L)}(p)+m}(\xi(x,p))
= & \sigma_{{\frak L}_2}^{m}(\psi_0(x))
= \psi_0(\sigma_{{\frak L}_1}^{m}(x)), \label{eq:sigma2xpsi0}\\
\sigma_{\widetilde{\frak L}_2}^{g_{(z,L)}(q)+m'}(\xi(z,q))
= &\sigma_{{\frak L}_2}^{m'}(\psi_0(z))
= \psi_0(\sigma_{{\frak L}_1}^{m'}(z)). \label{eq:sigma2zpsi0}
\end{align}
Since
$\sigma_{{\frak L}_1}^{m}(x)
=\sigma_{{\frak L}_1}^{m'}(z)$,
we have
$\sigma_{\widetilde{\frak L}_2}^{g_{(x,L)}(p)+m}(\xi(x,p))
=\sigma_{\widetilde{\frak L}_2}^{g_{(z,L)}(q)+m'}(\xi(z,q))
$
so that 
the element
\begin{equation*}
(\xi(x,p), g_{(x,L)}(p) +m - g_{(z,L)}(q) -m', \xi(z,q))
\end{equation*}
belongs to $G_{\widetilde{\frak L}_2}.$
Therefore  
the map
$\tilde{\varphi}:G_{\widetilde{\frak L}_1}
\longrightarrow  G_{\widetilde{\frak L}_2}
$
is well-defined.

The inverse of 
$\tilde{\varphi}:G_{\widetilde{\frak L}_1}
\longrightarrow  G_{\widetilde{\frak L}_2}
$
is given by the following way.
For $(y,N,w) \in G_{\widetilde{\frak L}_2}$,
one may find
$(x,p), (z,q) \in X_{{\frak L}_1}\times\Zp$
such that 
$y =\xi(x,p), w = \xi(z,q)$ in a unique way
because 
$\xi:X_{{\frak L}_1}\times\Zp \longrightarrow
X_{{\frak L}_2}\times\Zp
$ 
is a homeomorphism. 
Take $K,M\in \Zp$ 
such that
$N = K-M$ and
$$
\sigma_{\widetilde{\frak L}_2}^K(\xi(x,p))
=
\sigma_{\widetilde{\frak L}_2}^M(\xi(z,q)).
$$
We may assume that 
$K \ge g_{(x,L)}(p),\, M \ge g_{(z,L)}(q)$.
Since
\begin{align*}
\sigma_{\widetilde{\frak L}_2}^{K}(\xi(x,p))
= & \sigma_{{\frak L}_2}^{K-g_{(x,L)}(p)}(\psi_0(x))
= \psi_0(\sigma_{{\frak L}_1}^{K-g_{(x,L)}(p)}(x)), \\
\sigma_{\widetilde{\frak L}_2}^{M}(\xi(z,q))
= &\sigma_{{\frak L}_2}^{M-g_{(z,L)}(q)}(\psi_0(z))
= \psi_0(\sigma_{{\frak L}_1}^{M-g_{(z,L)}(q)}(z)),
\end{align*}
we have
\begin{equation*}
\sigma_{{\frak L}_1}^l(\sigma_{{\frak L}_1}^{K-g_{(x,L)}(p)}(x))
=
\sigma_{{\frak L}_1}^l(\sigma_{{\frak L}_1}^{M-g_{(z,L)}(q)}(z))
\end{equation*}
so that 
$
\sigma_{{\frak L}_1}^{l+K-g_{(x,L)}(p)}(x)
=
\sigma_{{\frak L}_1}^{l+M-g_{(z,L)}(q)}(z).
$
We then have
\begin{equation*}
\sigma_{\widetilde{\frak L}_1}^{l+K-g_{(x,L)}(p) +p}(e_{[p]}x)
=
\sigma_{\widetilde{\frak L}_1}^{l+M-g_{(z,L)}(q)+q}(e_{[q]}z).
\end{equation*}
Put
\begin{equation*}
n =(l+K-g_{(x,L)}(p) +p) -(l+M-g_{(z,L)}(q)+q) = N-g_{(x,L)}(p) +p+g_{(z,L)}(p) -q
\end{equation*}
so that 
$((x,p), n ,(z,q)) \in G_{\widetilde{\frak L}_1}$
and 
$n+q+g_{(x,L)}(p) -p-g_{(z,L)}(p) =K-M=N.$
This shows that 
$$
\tilde{\varphi}((x,p),n,(z,q)) = (y,N,w)
$$
proving that
 the map 
$\tilde{\varphi}:G_{\widetilde{\frak L}_1}
\longrightarrow  G_{\widetilde{\frak L}_2}
$
yields an isomorphism of the  \'etale groupoids.

 Recall that the factor map
$$
\tilde{\pi}_{{\frak{L}}_i}: G_{\widetilde{\frak L}_i}^0 = X_{{\frak L}_i} \times \Zp
\longrightarrow 
\widetilde{X}_{\Lambda_i} = X_{\Lambda_i} \times \Zp, \qquad i=1,2
$$
is defined in \eqref{eq:pitilde}. 
Define 
$\tilde{h}:\widetilde{X}_{\Lambda_1} \longrightarrow \widetilde{X}_{\Lambda_2}
$ by setting
$\tilde{h}(x,n) = \xi(x,n)$ for $(x,n) \in \widetilde{X}_{\Lambda_1}.$
It is a homeomorphism by Lemma \ref{lem:7.6}
and satisfies
$\tilde{\pi}_{{\frak L}_2} 
\circ \tilde{\varphi}|_{G_{\widetilde{\frak L}_1}^0} = 
\tilde{h}\circ \tilde{\pi}_{{\frak L}_1}.
$

We will next show  the equality
$c_{\widetilde{\frak L}_1} =c_{\widetilde{\frak L}_2}\circ\tilde{\varphi}. 
$
For an element
$ ((x,p), n,(z,q))\in G_{\widetilde{\frak L}_1},
$
the definition 
\eqref{eq:ctildel}
of 
$c_{{\widetilde{\frak L}}_1}$
tells us that 
$c_{{\widetilde{\frak L}}_1}((x,p), n,(z,q)) = n+q-p.$
On the other hand,
we see that
\begin{align*}
(c_{\widetilde{\frak L}_2}\circ\tilde{\varphi})((x,p), n,(z,q)) 
= & c_{\widetilde{\frak L}_2}(\xi(x,p), n+ q+ g_{(x,L)}(p) -p- g_{(z,L)}(q), \xi(z,q)) \\
= & n+q-p
\end{align*}
proving
$
c_{\widetilde{\frak L}_1}=
c_{\widetilde{\frak L}_2}\circ\tilde{\varphi}.
$
We thus proved the implication (i) $\Longrightarrow$ (ii).

\medskip

We will second prove the implication
(ii) $\Longrightarrow$ (i):
Suppose that 
there exist an isomorphism
$\tilde{\varphi}: G_{\widetilde{\frak L}_1}\longrightarrow 
G_{\widetilde{\frak L}_2}
$ of \'etale groupoids 
and a homeomorphism
$\tilde{h} :\widetilde{X}_{\Lambda_1} \longrightarrow \widetilde{X}_{\Lambda_2}
$
such that 
$\tilde{\pi}_{{\frak L}_2} 
\circ \tilde{\varphi}|_{G_{\widetilde{\frak L}_1}^0} = 
\tilde{h}\circ \tilde{\pi}_{{\frak L}_1}
$ 
and 
$c_{\widetilde{\frak L}_2}\circ\tilde{\varphi} =
c_{\widetilde{\frak L}_1}.$
For $x \in X_{{\frak L}_1}$, the element
$e_{[0]}x =(x,0)$ belongs to 
$ X_{\widetilde{\frak L}_1}$.
Hence
$((x,0),0,(x,0))$ defines an element of the groupoid
$G_{{\frak L}_1}$.
We then have
$
\tilde{\varphi}((x,0),0,(x,0)) =((\psi(x),m(x)),0,(\psi(x),m(x)))
$ 
for some $\psi(x) \in X_{{\frak L}_2}$ and $m(x) \in \Zp$.
Since $\tilde{\varphi}$ is continuous,
the maps 
$\psi:X_{{\frak L}_1}\longrightarrow X_{{\frak L}_2}$
and
$m: X_{{\frak L}_1}\longrightarrow \Z$
are continuous.
Since  
$((x,0), 1,(\sigma_{{\frak L}_1}(x),0))$ 
yields an element of the groupoid
$G_{{\frak L}_1}$ for $x \in X_{{\frak L}_1}$,
the condition
$c_{\widetilde{\frak L}_2}\circ\tilde{\varphi} =
c_{\widetilde{\frak L}_1}$
implies that
\begin{equation*}
\tilde{\varphi}((x,0), 1,(\sigma_{{\frak L}_1}(x),0))
=((\psi(x),m(x)), 1 + m(x) -m(\sigma_{{\frak L}_1}(x)),
 (\psi(\sigma_{{\frak L}_1}(x)), m(\sigma_{{\frak L}_1}(x))).
\end{equation*}
Hence there exist
$p,p' \in \Zp$ such that 
$p\ge m(x),\, p'\ge m(\sigma_{{\frak L}_1}(x))$ 
and
\begin{equation*}
p-p' =1+m(x)-m(\sigma_{{\frak L}_1}(x)),\qquad
\sigma_{{\frak L}_2}^{p-m(x)}(\psi(x)) 
= \sigma_{{\frak L}_2}^{p'-m(\sigma_{{\frak L}_1}(x))}(\psi(\sigma_{{\frak L}_1}(x))). 
\end{equation*}
By putting $l(x) = p'-m(\sigma_{{\frak L}_1}(x))\in \Zp$, we have
\begin{equation}
\sigma_{{\frak L}_2}^{l(x)+1}(\psi(x)) 
= \sigma_{{\frak L}_2}^{l(x)}(\psi(\sigma_{{\frak L}_1}(x))). \label{eq:sigmalpsix}
\end{equation}
As the function $l:X_{{\frak L}_1} \longrightarrow \Zp$
is continuous,
there exists $L \in \Zp$ such that 
\begin{equation*}
\sigma_{{\frak L}_2}^{L+1}(\psi(x)) 
= \sigma_{{\frak L}_2}^{L}(\psi(\sigma_{{\frak L}_1}(x)))
\quad
\text{ for }
x \in X_{{\frak L}_1}. 
\end{equation*}
Define
$\varphi =\sigma_{{\frak L}_2}^L\circ\psi:
X_{{\frak L}_1}\longrightarrow X_{{\frak L}_2}
$
which is continuous satisfying
 $
\sigma_{{\frak L}_2}\circ\varphi = \varphi\circ\sigma_{{\frak L}_1}.
$
Define 
$\bar{\varphi} :
\bar{X}_{{\frak L}_1}\longrightarrow \bar{X}_{{\frak L}_2}
$
by setting
\begin{equation}
\bar{\varphi}(x)_{[k,\infty)} = \varphi(x_{[k,\infty)})
\quad
\text{ for }
x =(x_n)_{n\in \Z}\in \bar{X}_{{\frak L}_1},\,
k\in \Z. \label{eq:varphix}
\end{equation}
By the condition
$\sigma_{{\frak L}_2}\circ\varphi =\varphi \circ\sigma_{{\frak L}_1}$,
we know that 
$\bar{\varphi}(x)$ defines an element of
$\bar{X}_{{\frak L}_2}$
and satisfies the condition 
 $
\bar{\sigma}_{{\frak L}_2}\circ\bar{\varphi}
 = 
\bar{\varphi}\circ\bar{\sigma}_{{\frak L}_1}.
$
We will show that 
$\bar{\varphi} :
\bar{X}_{{\frak L}_1}\longrightarrow \bar{X}_{{\frak L}_2}
$
is a homeomorphism.

We will first prove that 
$\bar{\varphi} :
\bar{X}_{{\frak L}_1}\longrightarrow \bar{X}_{{\frak L}_2}
$
is injective.
Now suppose that 
$\bar{\varphi}(x) = \bar{\varphi}(z)$
for $x,z \in  \bar{X}_{{\frak L}_1}.$
This implies that 
$\bar{\varphi}(x)_{[k,\infty)} = \bar{\varphi}(z)_{[k,\infty)}$
and hence
$\varphi(x_{[k,\infty)}) = \varphi(z_{[k,\infty)})$
for all $k \in \Zp$.
We first consider $k=1$ and put
$x' = x_{[1,\infty)}, \, 
 z' = z_{[1,\infty)}
\in X_{{\frak L}_1}.$
Hence we know
 that 
$\varphi(x') = \varphi(z').$
There exist $m,m'\in \Zp$ such that 
\begin{align*}
\tilde{\varphi}((x',0), 0, (x',0))
= & ((\psi(x'),m), 0, (\psi(x'), m)), \\
\tilde{\varphi}((z',0), 0, (z',0))
= & ((\psi(z'),m'), 0, (\psi(z'), m')).
\end{align*}
Since
$$
\sigma_{\widetilde{\frak L}_2}^{L+m}(\psi(x'), m)
=\sigma_{{\frak L}_2}^L(\psi(x')) =\sigma_{{\frak L}_2}^L(\psi(z')) 
=\sigma_{\widetilde{\frak L}_2}^{L+m'}(\psi(z'), m'),
$$
we know 
$
((\psi(x'),m), m-m', (\psi(z'), m')) \in G_{\widetilde{\frak L}_2}
$
so that 
$
\tilde{\varphi}^{-1}((\psi(x'),m), m-m', (\psi(z'), m')) \in G_{\widetilde{\frak L}_1}.
$
Since
$
\tilde{\varphi}^{-1}((\psi(x'),m), m-m', (\psi(z'), m')) 
= ((x',0), j, (z',0)) 
$
for some $j \in \Zp$,
we have
\begin{equation*}
\tilde{\varphi}((x',0), j, (z',0))
=((\psi(x'),m), m-m', (\psi(z'), m')). 
 \end{equation*}
By the assumption
$c_{\widetilde{\frak L}_1} = c_{\widetilde{\frak L}_2}\circ \tilde{\varphi}$,
we have
$$
j = 
c_{\widetilde{\frak L}_1}((x',0), j, (z',0))
=
c_{\widetilde{\frak L}_2}((\psi(x'),m), m-m', (\psi(z'), m'))
=0.
$$
Therefore we have
\begin{equation*}
\tilde{\varphi}^{-1}((\psi(x'),m), m-m', (\psi(z'), m')) 
=((x',0), 0, (z',0)).
\end{equation*}
It follows that 
there exists $k' \in \Zp$ such that 
$
\sigma_{\widetilde{\frak L}_1}^{k'}(x',0) 
=
\sigma_{\widetilde{\frak L}_1}^{k'}(z',0) 
$
so that 
$
\sigma_{{\frak L}_1}^{k'}(x') 
=
\sigma_{{\frak L}_1}^{k'}(z'). 
$
Let $k'(x',z')$ be the smallest such nonnegative integer $k'$ satisfying
$
\sigma_{{\frak L}_1}^{k'}(x') 
=
\sigma_{{\frak L}_1}^{k'}(z'). 
$
As in the proof of Theorem 5.1 (I) $\Longrightarrow$ (III)
in Carlsen--Rout's paper \cite{CR}, 
one knows that the function 
$k': \{(x', z') \in X_{{\frak L}_1}\times X_{{\frak L}_1} \mid 
\varphi(x') = \varphi(z') \} \longrightarrow \Zp
$
is continuous on the compact subset
$\{(x', z') \in X_{{\frak L}_1}\times X_{{\frak L}_1} \mid 
\varphi(x') = \varphi(z') \}$
of
$X_{{\frak L}_1}\times X_{{\frak L}_1}$.
Hence there exists 
$K \in \Zp$ such that 
$\sigma_{{\frak L}_1}^K(x') = \sigma_{{\frak L}_1}^K(z'),$
so that 
we have 
$x_{[1+K, \infty)} =z_{[1+K, \infty)}$
for all
$x', z' \in X_{{\frak L}_1}
$
satisfying 
$ 
\varphi(x') = \varphi(z').$
Now the condition 
$\bar{\varphi}(x) = \bar{\varphi}(z)$
also implies
$\varphi(\sigma_{{\frak L}_1}^{k-1}(x)_{[1,\infty)})
=\varphi(\sigma_{{\frak L}_1}^{k-1}(z)_{[1,\infty)})
$
for an arbitrary fixed $k \in \Z$.
Applying the above discussion for 
$\sigma_{{\frak L}_1}^{k-1}(x)_{[1,\infty)},
\, \sigma_{{\frak L}_1}^{k-1}(z)_{[1,\infty)}
$
instead of 
$x_{[1,\infty)}, \, z_{[1,\infty)},$
respectively,
we see that 
$\sigma_{{\frak L}_1}^{k-1}(x)_{[1+K,\infty)}
=\sigma_{{\frak L}_1}^{k-1}(z)_{[1+K,\infty)}
$
and hence
$x_{[k+K,\infty)} = z_{[k+K,\infty)}$.
As $k \in \Z$ is arbitrary, 
we have
$x = z$, and hence $\bar{\varphi}$ is injective.

We will next show that $\bar{\varphi}$ is surjective.
For $y \in X_{{\frak L}_2}$, we have
$
\tilde{\varphi}^{-1}((y,0),0,(y,0)) 
=((x,n),0,(x,n))
$
for some
$x \in X_{{\frak L}_1}$ and $n \in \Zp$.
Take $m \in \Zp$ such that 
$
\tilde{\varphi}((x,0),0,(x,0))
=((\psi(x),m),0,(\psi(x),m)).
$
We know that 
$((x,0), -n,(x,n)) \in G_{\widetilde{\frak L}_1}.$
Hence
$\tilde{\varphi}((x,0), -n,(x,n))
= ((\psi(x), m), q, (y,0))$
for some
$q \in \Z$.
Since
$c_{\widetilde{\frak L}_1} = c_{\widetilde{\frak L}_2}\circ\tilde{\varphi},$
we have
\begin{gather*}
c_{{\frak L}_1}((x,0), -n,(x,n)) = -n-(0-n) =0, \\
(c_{{\frak L}_2}\circ\tilde{\varphi})((x,0), -n,(x,n)) =((\psi(x), m), q, (y,0))= q-m,
\end{gather*}
so that $q=m$.
Hence we have
$\tilde{\varphi}((x,0), -n,(x,n))
= ((\psi(x), m), m, (y,0)).$
There exist
$p,q \in \Zp$ 
such that 
$p-q =m$
and
$\sigma_{{\frak L}_2}^p(\psi(x), m) =
\sigma_{{\frak L}_2}^q(y, 0),
$
so that 
$
\sigma_{{\frak L}_2}^{q}(\psi(x)) = 
\sigma_{{\frak L}_2}^q(y).
$
By a similar argument to the previous one,
we may find a positive integer $M\in \Zp$ such that 
for any $y \in X_{{\frak L}_2}$, there exists $x \in X_{{\frak L}_1}$ such that 
$
\sigma_{{\frak L}_2}^M(\psi(x)) = 
\sigma_{{\frak L}_2}^M(y).
$
It is routine to see that 
$\bar{\varphi}:\bar{X}_{{\frak L}_1}\longrightarrow \bar{X}_{{\frak L}_2}
$
is surjective.

 Therefore $\bar{\varphi}:\bar{X}_{{\frak L}_1}\longrightarrow \bar{X}_{{\frak L}_2}
$
is a homeomorphism satisfying
$\bar{\varphi}\circ \bar{\sigma}_{{\frak L}_1} =
\bar{\sigma}_{{\frak L}_2}\circ\bar{\varphi}.$
By the construction of 
$\bar{\varphi}$ given in \eqref{eq:varphix},
it is obvious that 
$\bar{\varphi}:\bar{X}_{{\frak L}_1}\longrightarrow \bar{X}_{{\frak L}_2}
$
is right asymptotically topologically conjugate.


We will next define a topological conjugacy 
$\bar{h}: \Lambda_1 \longrightarrow \Lambda_2$
in the following way.
For $\alpha = (\alpha_i)_{i\in \N} \in X_{\Lambda_1},$
the element 
$e_{[0]}\alpha = (\alpha,0)$ belongs to
$\widetilde{X}_{\Lambda_1}$,
so that 
$\tilde{h}(\alpha,0)$ belongs to
$\widetilde{X}_{\Lambda_2}$.
Hence one may find $h_1(\alpha) \in X_{\Lambda_2}$ and $p_1(\alpha) \in \Zp$ 
such that 
$\tilde{h}(\alpha,0) =(h_1(\alpha), p_1(\alpha)) \in \widetilde{X}_{\Lambda_2}.$
Since
$\tilde{h}:\widetilde{X}_{\Lambda_1}\longrightarrow \widetilde{X}_{\Lambda_2}$
is continuous,
so is the map $h_1:X_{\Lambda_1}\longrightarrow X_{\Lambda_2}$.
Take $x = (x_i)_{i\in \N} \in X_{{\frak L}_1}$ such that 
$\alpha = \pi_1(x)$.
Now by the equality
$\tilde{\pi}_{{\frak L}_2} 
\circ \tilde{\varphi}|_{G_{\widetilde{\frak L}_1}^0} = 
\tilde{h}\circ \tilde{\pi}_{{\frak L}_1},
$ 
we have
\begin{align*}
(h_1(\alpha), p_1(\alpha))
= & \tilde{h}(\alpha,0)  
=  \tilde{h}(\tilde{\pi}_{{\frak L}_1}(x,0))  
=  \tilde{\pi}_{{\frak L}_2}\circ\tilde{\varphi}(x,0) \\
= & \tilde{\pi}_{{\frak L}_2}(\psi(x), m(x))  
=  ({\pi}_2(\psi(x)), m(x)).
\end{align*}
Hence we have
$ h_1(\alpha) = \pi_2(\psi(x)).$
As
$
\sigma_{\Lambda_1}(\alpha) = 
\sigma_{\Lambda_1}(\pi_{1}(x)) = 
\pi_{1}(\sigma_{{\frak L}_1}(x)),
$
we have 
$$ h_1(\sigma_{\Lambda_1}(\alpha)) = 
{\pi}_2(\psi(\sigma_{{\frak L}_1}(x))).
$$
It then follows that 
\begin{align*}
\sigma_{\Lambda_2}^{L+1}(h_1(\alpha))
= & \sigma_{\Lambda_2}^{L+1}(\pi_2(\psi(x))) 
=  {\pi}_2(\sigma_{{\frak L}_2}^{L+1}(\psi(x))) 
=  {\pi}_2(\sigma_{{\frak L}_2}^{L}(\psi(\sigma_{{\frak L}_1}(x)))) \\
= & \sigma_{\Lambda_2}^{L}({\pi}_{2}(\psi(\sigma_{{\frak L}_1}(x)))) 
=  \sigma_{\Lambda_2}^{L}(h_1(\sigma_{\Lambda_1}(\alpha))).
\end{align*}
Define $h:X_{\Lambda_1}\longrightarrow X_{\Lambda_2}$ by
$h = \sigma_{\Lambda_2}^L\circ h_1$
and
$\bar{h}:{\Lambda_1}\longrightarrow {\Lambda_2}$ by
\begin{equation*}
\bar{h}(\alpha)_{[k,\infty)} = h(\alpha_{[k,\infty)}) 
\qquad
\text{ for } 
\alpha = (\alpha_i)_{i\in \Z} \in \Lambda_1, \, k \in \Z.
\end{equation*}
We then have a continuous map
$\bar{h}:{\Lambda_1}\longrightarrow {\Lambda_2}$ 
such that 
$\bar{h}\circ \sigma_{\Lambda_1} = \sigma_{\Lambda_2}\circ \bar{h}.$
Now it is obvious  by its construction that 
the equality
\begin{equation}
\bar{\pi}_2\circ \bar{\varphi} = \bar{h}\circ \bar{\pi}_1 \label{eq: barpivarphi1}
\end{equation}
holds. Since
both $\bar{\pi}_2, \bar{\varphi}$ are surjective,
we see that $\bar{h}:{\Lambda_1}\longrightarrow {\Lambda_2}$ 
is surjective.

Similarly we have a continuous surjection
$\bar{h'}:\Lambda_2 \longrightarrow \Lambda_1$ from $\tilde{h}^{-1}$ such that 
\begin{equation}
\bar{\pi}_1\circ \bar{\varphi}^{-1} = \bar{h'}\circ \bar{\pi}_2 \label{eq: barpivarphi2}.
\end{equation}
By \eqref{eq: barpivarphi1},
we have
$
\bar{h'}\circ\bar{\pi}_2\circ \bar{\varphi} 
= \bar{h'}\circ\bar{h}\circ \bar{\pi}_1 
$
so that 
\begin{equation*}
\bar{\pi}_1 
= \bar{h'}\circ\bar{h}\circ \bar{\pi}_1 
\end{equation*}
because of  \eqref{eq: barpivarphi2}.
As
$\bar{\pi}_1: \bar{X}_{{\frak L}_1} \longrightarrow \Lambda_1$
is surjective,
the map $\bar{h}: \Lambda_1 \longrightarrow \Lambda_2$ 
is injective and hence  a homeomorphism.
We obtain that
$(\Lambda_1, \sigma_1)$ and 
$(\Lambda_2, \sigma_2)$
are $({\frak L}_1,{\frak L}_2)$-conjugate.
Consequently, we  complete the proof of Theorem \ref{thm:twosidedconjugacy1}.
\qed

\medskip

{\it Acknowledgments:}
The author would like to deeply thank the referees 
for their careful reading, helpful suggestions and comments
for the first draft of the paper.
He particularly wishes to thank for helpful advices about Weyl groupoids, 
proof of Proposition \ref{prop:onesidedgroupoid} (ii) $\Longrightarrow$ (i)  
and proof of Theorem \ref{thm:twosidedconjugacy1} (ii) $\Longrightarrow$ (i).  
 This work was  supported by JSPS KAKENHI Grant Numbers 15K04896, 19K03537.




\end{document}